\documentclass[9pt,journal]{IEEEtran}
\ifCLASSOPTIONcompsoc
\usepackage[nocompress]{cite}
\else
\usepackage{cite}
\fi
\usepackage{hyperref} \hypersetup{colorlinks=true, linkcolor=green, breaklinks=false, urlcolor=blue}
\usepackage{graphicx}
\usepackage{float}
\usepackage{amssymb}
\usepackage{amsmath}
\usepackage{dsfont}
\usepackage{amsthm}
\usepackage{mathrsfs,doi,color}
\usepackage{empheq}
\usepackage{subfigure}
\usepackage{booktabs}
\usepackage{bm}
\usepackage{yhmath}
\usepackage{tabularx}
\usepackage{multirow}
\newtheoremstyle{mythm}{1.5ex plus 1ex minus .2ex}{1.5ex plus 1ex minus .2ex} {\rm}{\parindent}{\it\it}{\rm{:}}{1em}{}
\theoremstyle{mythm}

\renewcommand{\IEEEQED}{\IEEEQEDopen}

\newtheorem{theorem}{Theorem}[section]
\newtheorem{corollary}{Corollary}[section]
\newtheorem{lemma}{Lemma}[section]
\newtheorem{proposition}{Proposition}[section]

\newtheorem{remark}{Remark}

\ifCLASSINFOpdf
\else
\fi
\begin{document}

\title{An Optimal Pricing Formula for Smart Grid based on Stackelberg Game}

\author{Jiangjiang Cheng, Ge Chen, Zhouming Wu, and Yifen Mu
% \thanks{This paragraph of the first footnote will contain the date on
% which you submitted your paper for review. It will also contain support
% information, including sponsor and financial support acknowledgment. For
% example, ``This work was supported in part by the U.S. Department of
% Commerce under Grant BS123456.'' }
\thanks{This research is supported by the National Key Research and Development Program of China (2022YFA1004600), and the National Natural Science Foundation of China (12288201, 12071465).}
\thanks{Jiangjiang Cheng and Zhouming Wu are with the School of Mathematical Sciences, University of Chinese Academy of Sciences, Beijing 100049, China (e-mail: chengjiangjiang@amss.ac.cn; wuzhouming@amss.ac.cn).}
\thanks{Ge Chen and Yifen Mu are with the Key Laboratory of Systems and Control, Academy of Mathematics and Systems Science, Chinese Academy of Sciences, Beijing 100190, China (e-mail: chenge@amss.ac.cn; mu@amss.ac.cn).}
}

\IEEEtitleabstractindextext{%
  \begin{abstract}
	 %智能电网可以更加灵活的调控供需关系由于用户可以在响应到金融激励之后对能源消耗模式作出改变（reshape），也就是需求响应（DR）项目。实时定价方案（RTP）是最重要的需求响应策略之一，其中被效用公司发布的电价典型地以小时改变来实时调节用户的用电需求。在这篇文章中，我们将效用公司与用户之间的关系建模为一个Stackelberg博弈，并在不同条件下分别给出了相应的Stackelberg 均衡的解析解与数值解，即最优的RTP方案与相应的用电需求均衡，仿真实验验证了相应结果的最优性。
      The dynamic pricing of electricity is one of the most crucial demand response (DR) strategies in smart grid, where the utility company typically adjust electricity prices to influence user electricity demand.
      This paper models the relationship between the utility company and flexible electricity users as a Stackelberg game.
      Based on this model, we present a series of analytical results under certain conditions. First, we give an analytical Stackelberg equilibrium, namely the optimal pricing formula for utility company, as well as the unique and strict Nash equilibrium for users' electricity demand under this pricing scheme.  To our best knowledge, it is the first optimal pricing formula in the research of price-based DR strategies.
      Also, if there exist prediction errors for the supply and demand of electricity, we provide an analytical expression for the energy supply cost of utility company. 
     Moreover, a sufficient condition has been proposed that all electricity demands can be supplied by renewable energy.
      When the conditions for analytical results are not met, we provide a numerical solution algorithm for the Stackelberg equilibrium and verify its efficiency by simulation.
\end{abstract}

}

\maketitle

\IEEEdisplaynontitleabstractindextext

\IEEEpeerreviewmaketitle

\renewcommand{\thesection}{\Roman{section}}
\section{Introduction\label{sec1}}
\renewcommand{\thesection}{\arabic{section}}

% The advancement of renewable energy technologies has played a crucial role in addressing the challenges of greenhouse gas emissions and environmental pollution, making significant contributions to the sustainable development of our planet. However, the instability of renewable energy generation prevents it from dominating the energy supply in electricity markets \cite{raheli2021optimal, yap2019virtual}.
% %智能电网名词综述
% Consequently, by utilizing the latest information and control technologies, the smart grid faces a significant challenge with regard to regulating supply-demand balance \cite{rathor2020energy, ourahou2020review, gharbi2023demand}.
% bu2012smart, robu2019consider, salvo2018electrical, erdinc2014smart, , singh2010grid, white2018inaccurate
Smart grid refers to an advanced electricity distribution system that incorporates smart metering infrastructure capable of monitoring and measuring power consumption of users with sophisticated communication network technology \cite{gupta2015gadgets, beyea2010smart}, to achieve an efficient, reliable, and sustainable electric grid operation. \cite{smith2022effect, moslehi2010reliability}.
One of the key features of smart grid is the ability to regulate users' electricity demand with the aim of balancing supply and demand as well as reducing power generation costs, known as demand response (DR) \cite{siano2014demand, liu2023climate, li2021coordinating}. DR commonly refers to the energy consumption changes of users in response to factors such as electricity prices or incentive payments \cite{wang2023incentive, xu2021hybrid}, which plays a crucial role in optimizing energy allocation, alleviating peak load, and ensuring the stability of the power grid \cite{jiang2022multi, parizy2018low, zhong2016stability}.
% It serves as a critical mechanism for ensuring more efficient and reliable grid operations.
%需求响应名词综述、可再生能源的重要性、需求响应和可再生能源结合、主流方法是各种基于价格的需求响应用电
On the other hand, the development of renewable energy technology has a significant effect on addressing the challenges of greenhouse gas emissions and environmental pollution, gradually occupying a dominant position in the energy supply of power grid \cite{smith2022effect, sterl2020smart}.
Therefore, it is natural to design DR mechanisms that take into account renewable energy to achieve the goal of sustainable development \cite{cecati2011combined, dai2021real}.
% making significant contributions to the goal of sustainable development.
% However, the instability of renewable energy generation prevents it from dominating the energy supply in electricity markets.

% The DR mechanism regarding renewable energy generally refers to the changes in electricity demand for traditional and renewable energy when users respond to electricity prices or incentive benefits.
% 网络博弈在电网中的工程研究很多，理论研究凤毛麟角 网络博弈本身的理论研究也很困难，将其理论应用到电网则更极具挑战
Among various DR strategies, price-based DR strategy is a prevalent research topic, including day-ahead pricing, real-time pricing (RTP) and peak load pricing, etc \cite{qian2013demand, li2020real,mahmoudi2014wind}. In terms of improving the performance of energy market, RTP is the hottest area of DR mechanism, in which the electricity prices announced by the utility company change typically to regulate user electricity demand \cite{ding2020tracking, samadi2014real}.
% The electricity pricing formula (EPF) for regulating real-time prices is the hottest area of DR mechanism, %介绍实时定价相关工作
The study of RTP mechanism can help the smart grid achieve different objectives, including maximizing the social welfare \cite{namerikawa2015real}, minimizing the energy expenses \cite{atzeni2014noncooperative}, mitigating peak load \cite{baniasadi2018optimal}, and maintaining the stability of energy networks \cite{choi2017practical}.
Through the dynamic change in electricity prices, RTP incentivizes users to adjust their electricity consumption patterns to match real-time electricity supply, in order to realize more flexible and sustainable power system operation.

Because of the coupling relationship between multiple interactive entities in smart grid, such as the energy service provider and various users, network game theory provides a natural framework for the interaction among various participators with different objectives \cite{galeotti2010network, chen2023convergence, cheng2023evolutionary, riehl2016towards}. Most network game models typically exhibit complicated dynamic behaviors, making the theoretical analysis of these models very difficult \cite{madeo2014game, tan2016analysis}. Current mainstream methods include mean field dynamics \cite{moon2016linear, morimoto2015subsidy}, semi-tensor product method \cite{guo2013algebraic, cheng2016decomposed}, and so on. Zhu et al. \cite{zhu2016evolutionary} formulated and solved a demand-side management problem of networked smart grid based on evolutionary game theory. By using semi-tensor product, they proposed a new binary optimal control algorithm to minimize the overall cost of smart grid. Based on the randomness and dynamicity of wind and solar power generation, it is natural to employ a richer class of network games, namely stochastic game, to capture such dynamic uncertain environments.
Etesami et al. \cite{etesami2018stochastic} introduced a novel model for energy trading in smart grid employing a stochastic game framework on the demand side and an optimization problem on the utility side. They innovatively proposed a distributed algorithm to obtain equilibrium points of the system and adopted online convex optimization to solve the utility side problem. Besides, there are numerous network game models used to assist utility company in estimating energy demand and supply costs \cite{mohsenian2010autonomous}, achieving load balancing \cite{bayram2014unsplittable}, and managing energy storage and distributed microgrids \cite{li2022noncooperative}.
However, due to the large-scale, distributed, multi-agent nature of smart grid, and the uncertainty of renewable energy, the equilibrium analysis and solution of network game models in smart grid are still highly challenging, which will become an important direction for future endeavors.
% Considering realistic background, applying the theoretical tool of network game to smart grid is extremely challenging. 介绍网络博弈应用于电网相关工作

% Wen et al. \cite{wen2022demand} provided a dynamic price-based DR model and investigated its performance for the demand side management in smart grid. They explored the interaction between energy service providers and electricity consumers through game theory and proved the existence of Nash equilibrium. Tao and Gao \cite{tao2020real} formulated the related real-time pricing as a noncooperative game model with integration of distributed energy and storage devices. Based on dual decomposition theory, the existence of the optimal strategies in this model was analyzed and a distributed algorithm was further proposed to obtain the corresponding Nash equilibrium.  Goudarzi et al. \cite{goudarzi2021game} proposed a game theory-based interactive DR for handling dynamic prices in security-constrained electricity markets, and the results of the case studies demonstrate that this mechanism can achieve a win-win situation for the utility company and customers.
% In the modeling and optimization of various types of DR problems, game theory has emerged as a mainstream methodology.
%博弈论在价格中的应用

% 以及Given the typical hierarchical structure of the electricity market, 引出双层博弈
Due to the typical hierarchical structure inherent in the electricity market, numerous price-based DR problems with two decision levels can be effectively modeled as a special type of network game in smart grid -- Stackelberg game.
% As mentioned earlier, employing the Stackelberg game method, EPF has been extensively studied as an instance of two-level decision problems for many years.
The Stackelberg game model can be used to analyze the interactions between various entities in different scenarios.
A large amount of existing research focuses on the game relationship between utility company and users, and/or the interaction among users. Concretely, the utility company set dynamic electricity prices as a leader, and users adapt their electricity consumption patterns as followers, thereby achieving coordination of energy consumption and load distribution \cite{maharjan2013dependable, yu2015real}.
Additionally, the competition between multiple power utilities has gradually attracted widespread attention, with the market regulator setting price schemes to improve social welfare, and other power utilities adjusting pricing strategies accordingly to maximize their own profits while considering market competition \cite{yan2020distribution, li2019stackelberg}.

Although there have been many studies on price-based DR strategies in smart grid, a fundamental issue for utility companies is still open: what is the optimal pricing formula? There are some numerical solutions to this problem \cite{qian2013demand, samadi2014real, namerikawa2015real}, but its analytical expression has not yet been provided. This paper addresses this problem and its main contributions are as follows:

% Dai et al. \cite{dai2017real} studied the real-time pricing scheme in smart grid with multiple retailers and multiple residential users using Stackelberg game, and they proved the existence of Stackelberg equilibrium. Meng and Zeng \cite{meng2013stackelberg} proposed a Stackelberg game approach to maximize the profit of the electricity retailer and minimize the payment bills of its customers, and they adopted genetic algorithms to obtain the Stackelberg solutions. Li et al. \cite{li2023energy} proposed an energy trading management method based on Stackelberg game real-time pricing mechanism, which can solve complex optimization operation problems of microgrids. Yu et al. \cite{yu2016supply} presented a price-based Stackelberg game to model the interactions between the utility company and users, which aims to balance supply and demand as well as flatten the aggregated loads, and an iterative algorithm is proposed to derive the corresponding Stackelberg equilibrium.

% this study aims to find the optimal price formula that sufficiently utilizes renewable energy while meeting customer demands, minimizing the cost of controllable power generation, and maximizing social benefits.
% the uncertainty in the generation of renewable energy is beyond the scope of this paper.
\begin{itemize}
     \item This paper proposes a novel price-based DR model of Stackelberg game between the electricity company and users, in which the goal of the electricity company is to minimize the power supply cost, while the goals of users are to minimize electricity consumption costs. Different from previous works \cite{jiang2022multi, maharjan2013dependable, yu2015real}, our proposed model comprehensively considers the uncontrollability of wind and solar energy, the diversity of electricity users, and some constraints on users' electricity consumption, which makes it more realistic in certain scenarios.
     \item Based on our proposed model, we present a series of analytical results under certain inequality conditions. First, if the supply and demand of electricity can be predicted exactly in a certain time window, we give an analytical expression for the Stackelberg equilibrium, in which the unique strict Nash equilibrium among users is also the optimal solution for the utility company, therefore the price expression in the Stackelberg equilibrium can be treated as an optimal pricing formula. To the best of our knowledge, it is the first optimal pricing formula in the research of price-based DR strategies.

     Also, if there exist prediction errors for the supply and demand of electricity, we provide an analytical expression for the Nash equilibrium among users and the corresponding energy supply cost of utility company. Moreover, we find that under the same prediction accuracy, by shortening the sampling length of the time slot, the utility company's cost function under the Nash equilibrium can be greatly reduced.

     Based on above results, we further provide a sufficient condition for an ideal scenario where all electricity demand can be supplied by renewable energy sources such as wind and photovoltaic power.
     % Through this pricing formula for regulating real-time prices, the flexible users' electricity demand patterns can reach the unique and strict Nash equilibrium, and the uncontrollable energy can be sufficiently utilized, while the controllable energy generation cost of utility company can also be minimized. Based on these results, we further provide the conditions for an ideal scenario where all electricity demand can be supplied by renewable energy sources such as wind and photovoltaic power.
     % As an analytical Stackelberg equilibrium, an optimal pricing formula for integrating renewable energy is determined, which can achieve a unique Nash equilibrium in users' electricity demand while minimizing the cost of traditional energy generation.  minimizing the cost of controllable power generation, and maximizing social benefits.%考虑可再生的DR
     \item When the inequality conditions for the analytical solution are not satisfied, 
     a search algorithm is developed to obtain the numerical solution of Stackelberg equilibrium,  and a simulation example is presented to demonstrate the effectiveness of our proposed algorithm.
     % Under certain conditions, a numerical solution algorithm is developed to determine the Stackelberg equilibrium. Additionally, a simulation example is presented to demonstrate the effectiveness of the algorithm.
\end{itemize}
% Specifically, we model the interaction between utility company and flexible electricity users as a bi-level Stackelberg game model. By solving the game problem, we obtain the optimal price scheme for the utility company, as well as the Nash equilibrium for users' electricity demand under this pricing scheme.
% That is to say, we derive the analytical solution for the Stackelberg equilibrium of this model under certain conditions. When the conditions are not met, we provide the numerical solution algorithm for the corresponding Stackelberg equilibrium and show a simulation example.

The rest of this paper is organized as follows. Section \ref{sec2} introduces the constructed Stackelberg game model and some assumptions. Section \ref{sec3} analyzes the proposed game model and provides analytical solutions for the corresponding Stackelberg equilibrium under certain conditions, and the proofs of the corresponding theoretical results are presented in Section \ref{sec4}.
When the conditions are not met, Section \ref{sec5} provides a numerical solution algorithm for Stackelberg equilibrium. The simulation results are presented and discussed subsequently. Finally, the conclusion and outlook are made in Section \ref{sec6}.

\renewcommand{\thesection}{\Roman{section}}
\section{Problem Formulation and Assumptions\label{sec2}}
\renewcommand{\thesection}{\arabic{section}}

This section provides the formulation of the problem to be solved in this paper. Without loss of generality, our system includes one supply side, one service side and one demand side.
The supply side consists of various types of power supplies, which are divided into two categories: one is controllable power supply (e.g. thermal power generation, hydro power generation, storage);
the other one is uncontrollable power supply (e.g. wind power generation, photovoltaic power generation).
On the service side, there is a utility company engaged in electricity trading between the supply side and  demand side. By setting an appropriate price, the utility company can regulate users' electricity demands, and then feedback energy consumption information to supply side to achieve supply-demand balance. The demand side is composed of a large number of users, which can be divided into two categories: flexible electricity users (e.g. electric vehicles, hydrogen production plant, some hot water supply systems, some discrete electricity enterprises) and regular electricity users (e.g., household electricity, continuous electricity enterprises). The difference between these two categories of users is: the flexible users determine their electricity consumption according to the spot price, while the electricity demands  of regular users do not depend on the spot price.
Fig.~\ref{frame} illustrates the process of energy trading between power supplies, utility company, and users.
For ease of understanding, Table \ref{notations} summarizes the main notations used in this paper.
\begin{figure}[htp]
     \centering
     \includegraphics[width=3.4in]{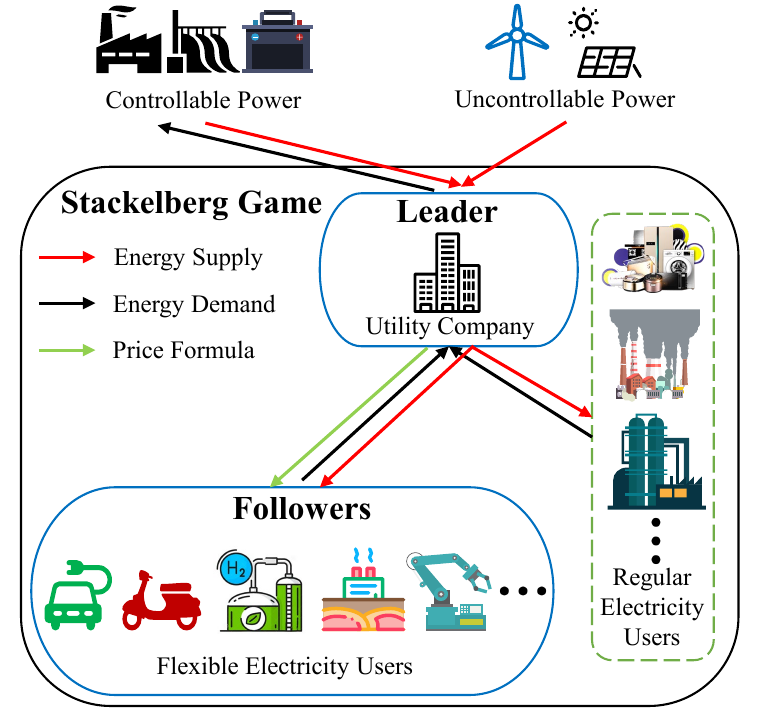}
     \caption{The process of energy trading between power supplies, utility company, and users.}
     \label{frame}
\end{figure}

\begin{table}[ht]
% \refstepcounter{table}
% \hypertarget{notations}{}
\renewcommand{\arraystretch}{1.8}
\begin{center}\footnotesize
\setlength{\tabcolsep}{12pt}
\caption{LIST OF MAIN NOTATIONS.}
\label{notations}
\begin{tabularx}{\columnwidth}{cX}
\toprule
Notation & \makebox[5.5cm][c]{Description} \\
\midrule
$\mathcal{N}$ & Set of all followers, i.e., $\mathcal{N}=\{1,2,\cdots,n\}$\\
$T$ & Number of scheduling time window\\
$\mathbf{1}$ & A column vector whose elements are all $1$\\
$\nu_i(\cdot)$ & Electricity demand of flexible user $i$ at each time slot\\
\multirow{2}*{$\bm{\nu}_i$} & The strategy of flexible user $i$, i.e., $\bm{\nu}_i=(\nu_i(1),\nu_i(2),\cdots,\nu_i(T))^{\top}$ \\
\multirow{2}*{$\bm{\nu}$} & The strategy combination of all flexible users, i.e., $\bm{\nu}=(\bm{\nu}_1,\bm{\nu}_2,\ldots,\bm{\nu}_n)^{\top}$\\
\multirow{2}*{$\bm{\nu}_{-i}$} & The strategies of all flexible users except the $i$th one, i.e., $\bm{\nu}_{-i}=(\bm{\nu}_1,\ldots,\bm{\nu}_{i-1},\bm{\nu}_{i+1},\ldots,\bm{\nu}_n)^{\top}$ \\
\multirow{2}*{$\nu_i^{\max}$} & Maximum electricity consumption of flexible user $i$ at each time slot\\
\multirow{2}*{$\nu_{\mathcal{N}}(\cdot)$} & Total electricity demand of all flexible users at each time slot, i.e., $\nu_{\mathcal{N}}(\cdot)=\sum_{i\in\mathcal{N}}\nu_i(\cdot)$\\
$g_i$ & Total electricity demand of flexible user $i$\\
\multirow{2}*{$g_{\mathcal{N}}$} & Total electricity demand of all flexible users, i.e., $g_{\mathcal{N}}=\sum_{i\in\mathcal{N}}g_i$\\
\multirow{2}*{$r(\cdot)$} & Total electricity demand of all regular users at each time slot\\
$w(\cdot)$ & Uncontrollable power generation at each time slot\\
$c(\cdot)$ & Controllable energy supply at each time slot\\
\multirow{2}*{$\pi(\cdot)$} & Electricity price function at each time slot, i.e., $\pi(\cdot)=\pi(\nu_{\mathcal{N}}(\cdot),r(\cdot),w(\cdot))$\\
$\bm{\pi}$ & Leader's strategy, i.e., $\bm{\pi}=(\pi(1),\pi(2),\ldots,\pi(T))^{\top}$\\
$u^l(\cdot)$ & The cost function of leader\\
$u_i^f(\cdot)$ & The electricity cost of flexible user $i$\\
\bottomrule
\end{tabularx}
\end{center}
\end{table}
%画图
%三类主体：发电厂和新能源发电站、效用公司、灵活用电（传统用电）用户

\subsection{Stackelberg Game Model}

We model the interaction between the utility company and flexible electricity users as a Stackelberg game with one leader and $n$ followers. The utility company is the leader who regulates users' energy demands by setting proper electricity prices, thereby reducing the cost of controllable power generation; the flexible users are the followers who determine their electricity demand patterns after obtaining corresponding price schemes, thereby minimizing their own electricity costs.
The mathematical representations of our proposed Stackelberg game are introduced as follows.

\emph{Demand side}: Let $\mathcal{N}:=\{1,2,\cdots,n\}$ be the set of $n$  followers, where each follower denotes a flexible electricity user.
Considering some practical scenarios, such as electric vehicle charging or electricity consumption of factories produced according to orders, we assume that
 each user $i\in\mathcal{N}$ has a total electricity demand $g_i>0$ that must be satisfied within $T$ time periods, where $T$ represents the scheduling time window.
To finish the energy consumption goals while minimizing their own electricity costs, each user $i\in\mathcal{N}$ needs to determine the electricity demand $\nu_i(t)$, $0\leq \nu_i(t)\leq \nu_i^{\max}$ at each time slot $t, t=1,\cdots,T$. Here $\nu_i^{\max}$ represents the maximum electricity consumption  of user $i$ at each time slot, which satisfies $\nu_i^{\max} \geq g_i/T$.
The strategy of user $i$ is set to be  $$\bm{\nu}_i:=(\nu_i(1),\nu_i(2),\cdots,\nu_i(T))^{\top}.$$ Therefore, the total electricity demand of all flexible electricity users at time slot $t$ is $\nu_{\mathcal{N}}(t):=\sum_{i\in\mathcal{N}}\nu_i(t)$, $t=1,\cdots,T$, and the total electricity demand of all flexible electricity users is $g_{\mathcal{N}}:=\sum_{i\in\mathcal{N}}g_i$.
Let $\bm{\nu}$ be the strategy combination of all flexible users, i.e., $\bm{\nu}=(\bm{\nu}_1,\bm{\nu}_2,\ldots,\bm{\nu}_n)^{\top}.$

On the other hand, the regular electricity users have their own electricity demands at each time slot and are not affected by real-time electricity prices. We assume that the total electricity demand of all regular users at time slot $t$ is $r(t)$, $t=1,\cdots,T$, which is assumed to be a known sequence.

\emph{Supply side}:  We denote the energy supply of controllable and uncontrollable energy at time slot $t$ as $c(t)$ and $w(t)$, respectively.
It is worth mentioning that the uncontrollable power generation $w(t)$ significantly depends on climate conditions and should be a stochastic sequence. However, to obtain some deterministic results
we assume $\{w(t)\}_{t=1,\ldots,T}$ is a deterministic sequence.
%%w(t)依赖于天气条件，这里我们把它理解为在已知当前天气条件下的预测值；

\emph{Service side}: As the leader, the utility company regulates the electricity demands of users by setting an appropriate electricity price $\pi(t)=\pi(\nu_{\mathcal{N}}(t),r(t),w(t))$, $t=1,\cdots,T$, which is a function with respect to $\nu_{\mathcal{N}}(t)$, $r(t)$, and $w(t)$. We denote the strategy of utility company as $$\bm{\pi}:=(\pi(1),\pi(2),\ldots,\pi(T))^{\top}.$$
The utility company need purchase $c(t)$, $t=1,\cdots,T$ electricity quantity
from controllable power generation to guarantee the dynamic balance between electricity demand and supply, i.e.,
$\nu_{\mathcal{N}}(t)+r(t)=w(t)+c(t)$. Since
\begin{equation*}
     \begin{aligned}
          \sum_{t=1}^T c(t)&=\sum_{t=1}^T \left[\nu_{\mathcal{N}}(t)+r(t)-w(t)\right]\\
                           &=\sum_{i=1}^n g_i+ \sum_{t=1}^T \left[r(t)-w(t)\right]\\
                           &=g_{\mathcal{N}}+ \sum_{t=1}^T \left[r(t)-w(t)\right]
     \end{aligned}
\end{equation*}
is assumed to be a constant, we set the cost function $u^l(\bm{\pi},\bm{\nu})$ of utility company to be the variance of controllable power $c(t)$.
%where $\bm{\nu}$ is the strategy combination for all flexible users, that is, $\bm{\nu}:=(\bm{\nu}_1,\bm{\nu}_2,\ldots,\bm{\nu}_n)^{\top}$.
Let $\bar{c}:=\frac{1}{T}\sum_{t=1}^T c(t)$ be the average value of $c(t)$ over the entire time periods.
Therefore, the leader-level optimization problem is formulated in \eqref{leader}.

For each flexible electricity user $i\in\mathcal{N}$, let $\bm{\nu}_{-i}$ denote the strategies of all users except the $i$th one.
The utility of flexible user $i$ is set to be its electricity cost $u_i^f(\bm{\pi},\bm{\nu}_i,\bm{\nu}_{-i})$,
which is a function with respect to the strategies of the utility company and all flexible users.
 Then,  the follower-level game among all flexible users can be described by \eqref{followers}.

\vskip 2mm
\begin{center}
\noindent\framebox[28em][r]{%
  \parbox{3.7in}{%
       \begin{align}
          &\textbf{Leader-Level Optimization:}\nonumber\\
          &\min_{\bm{\pi}}u^l(\bm{\pi},\bm{\nu})=\frac{1}{T}\sum_{t=1}^T\big[c(t)-\bar{c}\big]^2, \label{leader}\\
          &~~\mathrm{s.t.}~ \nu_{\mathcal{N}}(t)+r(t)=w(t)+c(t), ~~t=1,\cdots,T.\nonumber\\
          &\textbf{Follower-Level Game: }\mbox{For each flexible user }i=1,\ldots,n, \nonumber\\
          &\min_{\bm{\nu}_i}u_{i}^f(\bm{\pi},\bm{\nu}_i,\bm{\nu}_{-i})=\sum_{t=1}^T \pi(\nu_{\mathcal{N}}(t),r(t),w(t))\cdot \nu_i(t), \nonumber\\
          &~~\mathrm{s.t.}~ \sum_{t=1}^T \nu_i(t)=g_i,  \label{followers}\\
          &~~~~~~ 0\leq \nu_i(t)\leq  \nu_i^{\max},  ~~t=1,\cdots,T.\nonumber
     \end{align}
  }%
}
\end{center}
\vskip 2mm

\emph{Overall process of the Stackelberg Game}:
The utility company first announces the price formula to all flexible users. Then, each user selects the corresponding optimal demand strategy as the ``best response'' to the leader's strategy, and feeds its strategy back to the utility company. Therefore, the solution of the Stackelberg game problem can be ultimately transformed into the choice of the optimal pricing formula of the utility company,
based on the rational assumption that each flexible user selects its  best-response strategy.

\renewcommand{\thesection}{\Roman{section}}
\section{Analysis of Stackelberg Game Model\label{sec3}}
\renewcommand{\thesection}{\arabic{section}}

In this section, we analyze the proposed Stackelberg game \eqref{leader}-\eqref{followers}. First, we provide a strict definition of Stackelberg equilibrium as the solution of the Stackelberg game. Then, we present a series of analytical results for equilibria in Stackelberg game \eqref{leader}-\eqref{followers} under certain conditions.
Concretely, we give an analytical Stackelberg equilibrium, namely the optimal pricing formula for utility company, as well as the unique and strict Nash equilibrium for users' electricity demand under this pricing scheme. 
Also, if there exist prediction errors for the supply and demand of electricity, we provide an analytical expression for the energy supply cost of utility company. 
Moreover, a sufficient condition has been proposed that all electricity demands can be supplied by new energy sources.

% Then, we solve the two types of SGM in Section \ref{sec2}.

\subsection{Stackelberg Equilibrium}\label{sec3_1}

For the proposed Stackelberg game \eqref{leader}-\eqref{followers}, we introduce the classic Stackelberg equilibrium (SE) \cite{bacsar1998dynamic} formulated as follows.

Given a leader's strategy $\bm{\pi}$, the best response strategy $\bm{\nu}(\bm{\pi})=(\bm{\nu}_1(\bm{\pi}),\bm{\nu}_2(\bm{\pi}),\ldots,\bm{\nu}_n(\bm{\pi}))^{\top}$ of followers is the Nash equilibrium of game model \eqref{followers}, that is, for any user $i\in\mathcal{N}$ and any strategy $\bm{\nu}_i'$, we have
\begin{equation}\label{NEdef}
     u_i^f(\bm{\pi},\bm{\nu}_i',\bm{\nu}_{-i}(\bm{\pi}))\geq u_i^f(\bm{\pi},\bm{\nu}_{i}(\bm{\pi}),\bm{\nu}_{-i}(\bm{\pi})).
\end{equation}
Further, if the inequality in \eqref{NEdef} strictly holds for any $\bm{\nu}_i'\neq \bm{\nu}_{i}(\bm{\pi})$, we say that $\bm{\nu}(\bm{\pi})$ is the \emph{strict} Nash equilibrium.
Meanwhile, based on the assumption that the followers reach a Nash equilibrium,  the leader should try to minimize its cost function $u^l(\bm{\pi},\bm{\nu}(\bm{\pi}))$
by optimizing its price function $\bm{\pi}$. Let $\hat{\bm{\pi}}$ be the solution of the problem $\min_{\bm{\pi}}u^l(\bm{\pi},\bm{\nu}(\bm{\pi}))$, and $\hat{\bm{\nu}}=\bm{\nu}({\hat{\bm{\pi}}})$ be the best responses of all followers to the leader. Then $(\hat{\bm{\pi}}, \hat{\bm{\nu}})$ is the SE of the Stackelberg game \eqref{leader}-\eqref{followers}.

In addition, we say  $(\hat{\bm{\pi}}, \hat{\bm{\nu}})$ is a \emph{perfect SE} of the Stackelberg game \eqref{leader}-\eqref{followers} if $u^l(\hat{\bm{\pi}}, \hat{\bm{\nu}})=0$.

%by choosing optimal price function, As a result, the leader need solve the following optimization problem:
%\begin{equation}
%     \min_{\bm{\pi}}u^l(\bm{\pi},\bm{\nu}(\bm{\pi})).
%     \label{opt}
%\end{equation}
%Let $\bm{\pi}^\ast$ be the solution to the problem \eqref{opt}, and let $\bm{\nu}^\ast=\nu(\pi^\ast)$, then $(\pi^\ast, \nu^\ast)$ is the SE of Stackelberg game \eqref{leader}-\eqref{followers}.
%%SE的严格定义，SGM的流程详细说说

\subsection{Main Results}\label{ssmainresult}

A natural idea is to set the price function as the ratio of total electricity demand to power generation capacity. In order to fully utilize the power generation capacity of wind and solar power generation,
we define a price function as
\begin{equation}\label{priceform0}
\pi(t)=\frac{\nu_{\mathcal{N}}(t)+r(t)+a_2(t)}{w(t)+a_1(t)},~~\forall 1\leq t\leq T,
\end{equation}
where $a_1(t)$ and $a_2(t)$ are two adjustable factors. We aim to find the SE of the Stackelberg game \eqref{leader}-\eqref{followers} by adjusting the values of $a_1(t)$ and $a_2(t)$.
Set
\begin{equation}\label{wandr}
     \widetilde{w}(t):=w(t)+a_1(t),~ \widetilde{r}(t):=r(t)+a_2(t),~\forall 1\leq t\leq T.
\end{equation}
Under the leader's strategy \eqref{priceform0}, we first analyze the Nash equilibrium of the following game model:
\begin{eqnarray}\label{follower2}
\left\{
\begin{aligned}
 &\min_{\bm{\nu}_i}u_{i}^f(\bm{\pi},\bm{\nu}_i,\bm{\nu}_{-i}),\\
     &\mathrm{s.t.}~ \sum_{t=1}^T \nu_i(t)=g_i,
     \end{aligned}~~~~i=1,\ldots,n. \right.
\end{eqnarray}
% model \eqref{followers} without the inequality constraints under the leader's strategy \eqref{priceform0}.

\begin{proposition}\label{saddle}
   Assume that the leader's strategy adopts \eqref{priceform0}, and  $\widetilde{w}(t)>0$ for any $1\leq t\leq T$, then the strategy combination  $\bm{\nu}$  of all followers with
     \begin{equation}\label{saddpoi}
          \begin{aligned}
               \nu_i(t)&=\frac{\widetilde{w}(t)}{\sum_s\widetilde{w}(s)}g_i+\frac{1}{(n+1)\sum_s\widetilde{w}(s)}\\
               &\qquad\qquad\times\sum_{s=1}^T\widetilde{w}(s)\widetilde{w}(t)\left[\frac{\widetilde{r}(s)}{\widetilde{w}(s)}-\frac{\widetilde{r}(t)}{\widetilde{w}(t)}\right],\\
               &\qquad\qquad\qquad\qquad~ i=1,\ldots,n,~ t=1,\ldots,T
          \end{aligned}
     \end{equation}
     is a strict and unique Nash equilibrium of the game model \eqref{follower2}.
     % \begin{eqnarray}\label{follower2}
     % \left\{
     % \begin{aligned}
     %  &\min_{\bm{\nu}_i}u_{i}^f(\bm{\pi},\bm{\nu}_i,\bm{\nu}_{-i}),\\
     %      &\mathrm{s.t.}~ \sum_{t=1}^T \nu_i(t)=g_i,
     %      \end{aligned}~~~~i=1,\ldots,n. \right.
     % \end{eqnarray}
\end{proposition}

\vskip 2mm

For the model \eqref{followers}, we can prove that the Nash equilibrium is strict and unique, however its explicit expression is unknown:
\begin{proposition}\label{oristruni}
     Assume that the leader's strategy adopts \eqref{priceform0} with $\widetilde{w}(t)>0$ for any $1\leq t\leq T$, then the model \eqref{followers} admits a strict and unique Nash equilibrium.
\end{proposition}

For any $1\leq t\leq T$, if we choose
\begin{eqnarray}\label{cstar}
     \begin{aligned}
     &a_1^*(t):=\frac{g_{\mathcal{N}}}{T}-r(t)+\frac{1}{T}\sum_{s=1}^T[r(s)-w(s)],\\
     &a_2^*(t):=\frac{n+1}{n}[w(t)+a_1^*(t)]-r(t),
     \end{aligned}
\end{eqnarray}
then the price function \eqref{priceform0} becomes
\begin{align}\label{priceform}
&~~\pi^*(t):=\frac{\nu_{\mathcal{N}}(t)+r(t)+a_2^*(t)}{w(t)+a_1^*(t)}\\
&=1+\frac{1}{n}+\frac{\nu_{\mathcal{N}}(t)}{w(t)+\dfrac{g_{\mathcal{N}}}{T}-r(t)+\frac{1}{T}\sum_{s=1}^T[r(s)-w(s)]}.\nonumber
\end{align}
We set
\begin{equation}\label{thm1_p1}
     \begin{aligned}
          \widetilde{w}^*(t):&=w(t)+a_1^*(t)\\
          &=\frac{g_{\mathcal{N}}}{T}+w(t)-r(t)+\frac{1}{T}\sum_{s=1}^T[r(s)-w(s)],\\
          \widetilde{r}^*(t):&=r(t)+a_2^*(t)=\frac{n+1}{n}\widetilde{w}^*(t),~\forall 1\leq t\leq T.
     \end{aligned}
\end{equation}
Then, with \eqref{thm1_p1}, the strict and unique Nash equilibrium \eqref{saddpoi} of model \eqref{follower2} becomes
\begin{eqnarray}\label{NE2}
         && \nu_i^*(t):=\frac{\widetilde{w}^*(t)}{\sum_s\widetilde{w}^*(s)}g_i+\frac{1}{(n+1)\sum_s\widetilde{w}^*(s)}\nonumber\\
          &&~~~~~~~~~~~~~\quad\times\sum_{s=1}^T\widetilde{w}^*(s)\widetilde{w}^*(t)\left[\frac{\widetilde{r}^*(s)}{\widetilde{w}^*(s)}-\frac{\widetilde{r}^*(t)}{\widetilde{w}^*(t)}\right]\nonumber\\
          &&~~~~~=\frac{g_i}{\sum_s\widetilde{w}^*(s)}\widetilde{w}^*(t)=\frac{g_i}{g_{\mathcal{N}}}\widetilde{w}^*(t)\nonumber\\
          &&~~~~~=\frac{g_i}{T}+\frac{g_i}{g_{\mathcal{N}}}[w(t)-r(t)]+\frac{g_i}{g_{\mathcal{N}} T}\sum_{s=1}^T[r(s)-w(s)],\nonumber\\
          &&~~~~~\qquad\qquad\qquad~ i=1,\ldots,n,~ t=1,\ldots,T.
\end{eqnarray}
% Further, the strict Nash equilibrium \eqref{saddpoi} of model \eqref{followers} is
% \begin{equation}\label{saddpoi2}
%      \begin{aligned}
%           \nu_i^*(t):&=\frac{g_i}{T}+\frac{g_i}{g_{\mathcal{N}}T}\sum_{s=1}^T\left\{[r(s)-w(s)]-[r(t)-w(t)]\right\},\\
%                &=\frac{g_i}{T}+\frac{g_i}{g_{\mathcal{N}}}[w(t)-r(t)]+\frac{g_i}{g_{\mathcal{N}} T}\sum_{s=1}^T[r(s)-w(s)],\\
%                &\qquad\qquad\qquad\qquad\qquad~ i=1,\ldots,n, t=1,\ldots,T.
%      \end{aligned}
% \end{equation}

Let $\bm{\pi}^*:=(\pi^*(1),\pi^*(2),\ldots,\pi^*(T))^{\top}$ denote the leader's strategy,  $\bm{\nu}_i^*:=(\nu_i^*(1),\nu_i^*(2),\ldots,\nu_i^*(T))^{\top}$ denote the strategy of follower $i$,
and $\bm{\nu}^*:=(\bm{\nu}_1^*,\bm{\nu}_2^*,\ldots,\bm{\nu}_n^*)^{\top}$ denote the strategy combination of all followers. We have
\begin{proposition}\label{leaderop}
     When the leader's strategy is $\bm{\pi}^*$ and the followers' strategy is $\bm{\nu}^*$, the leader's cost function $u^l(\bm{\pi}^*,\bm{\nu}^*)=0$.
\end{proposition}

The proofs of Proposition \ref{saddle}-\ref{leaderop} are put in Subsection \ref{PTheorem0}.

Based on Proposition \ref{saddle}-\ref{leaderop}, we now demonstrate that the Stackelberg game \eqref{leader}-\eqref{followers} can achieve a perfect SE if the following inequalities hold:
\begin{multline}\label{asum1}
     0<\frac{g_{\mathcal{N}}}{T}+w(t)-r(t)+\frac{1}{T}\sum_{s=1}^T[r(s)-w(s)]=\widetilde{w}^*(t)\\
     \leq\min_{i\in\mathcal{N}}\left\{\frac{\nu_i^{\max}}{g_i}\right\}g_{\mathcal{N}},~~\forall 1\leq t\leq T.
\end{multline}
% Next we illuminate how perfect SE of Stackelberg game \eqref{leader}-\eqref{followers} is achieved.
We show that, under the condition (\ref{asum1}), $\bm{\nu}^*$ is the best-response strategy  to the leader's strategy $\bm{\pi}^*$, and they contribute a perfect SE of the Stackelberg game \eqref{leader}-\eqref{followers}.
Thus, $\pi^*(t)$ can be treated as an optimal pricing formula for the utility company.
% Similarly, for the SGM II satisfying assumption \ref{asum1}, the following theorem provides corresponding SE result.

\begin{theorem}[Perfect SE with optimal pricing formula $\pi^*(t)$]\label{thm1}
If the inequalities (\ref{asum1}) hold then\\
i) $\bm{\nu}^*$ is a strict and unique Nash equilibrium of model \eqref{followers} when
the leader's strategy is $\bm{\pi}^*$;\\
ii) $(\bm{\pi}^*,\bm{\nu}^*)$ is a perfect SE of the Stackelberg game \eqref{leader}-\eqref{followers}.
\end{theorem}

Fig.~\ref{Fig2} shows an example of the curves of $\{w(t),r(t),c(t)\}_{t=1,\ldots,T}$ and $\{\nu_i^*(t)\}_{i=1,\ldots,n;t=1,\ldots,T}$, where we make $n=10$, $T=24$ and they satisfy the inequality \eqref{asum1}, thus presenting a perfect SE profile.
\begin{figure}[htp]
     \centering
     \includegraphics[width=3.2in]{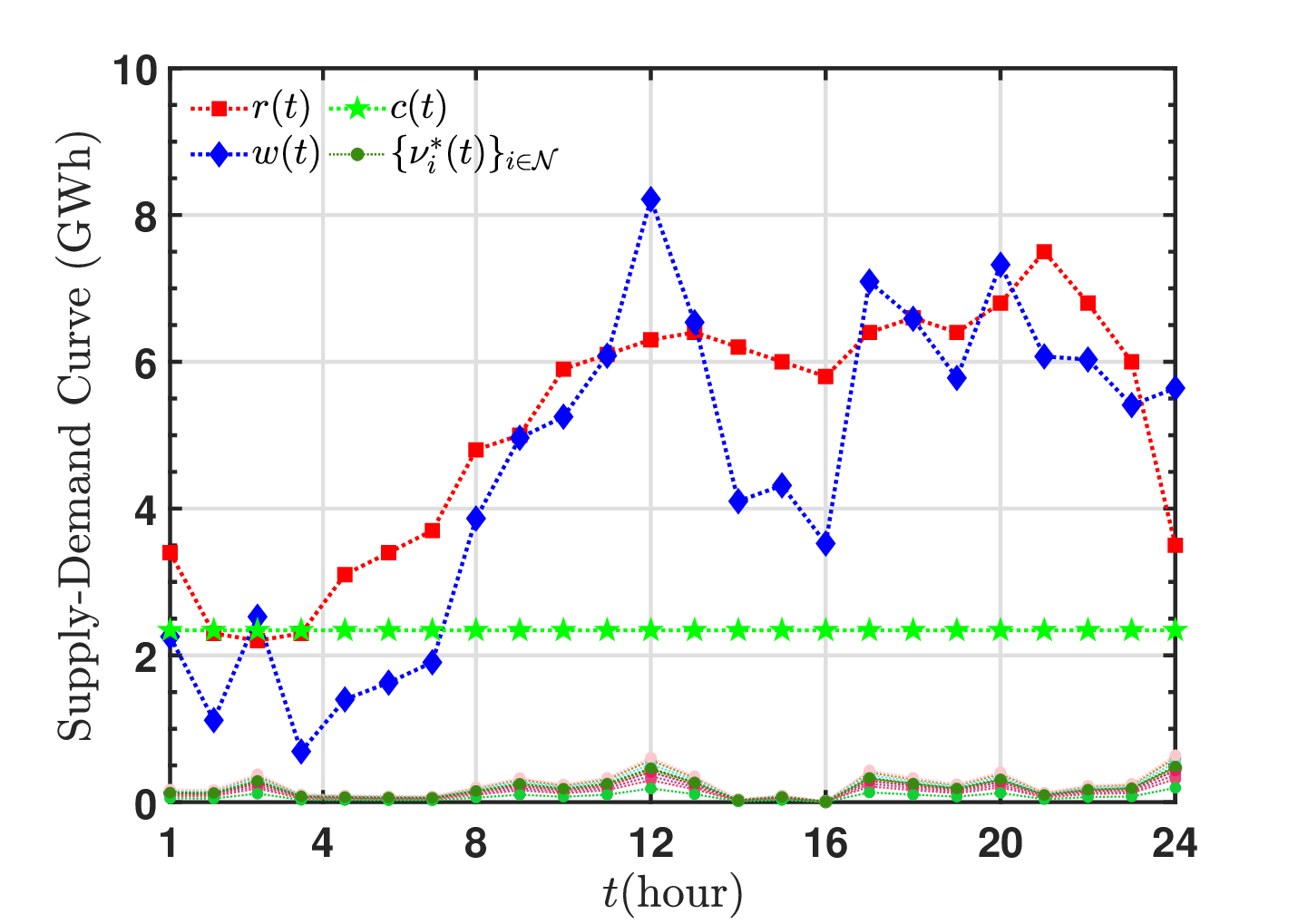}
     \caption{An example of achieving a perfect SE, in which the uncontrollable energy accounts for 65.84\% of energy supply, and flexible users account for 25.29\% of the total energy demand.}
     \label{Fig2}
\end{figure}

\begin{remark}
From Theorem \ref{thm1}, if the utility company adopts the price formula $\pi^*(t)$, all flexible users have  a stable electricity strategy combination $\bm{\nu}^*$, in which
each flexible user will not proactively change its strategy to avoid an increase in electricity cost. Meanwhile, all electricity needs are met, all new energy generation is utilized,
and the controllable power supply $c(t)\equiv const$, thereby social benefits are maximized.
\end{remark}

An interesting question is, if we ignore the stability factors of the grid and only consider the supply-demand balance of electricity, under what conditions can all electricity be provided by new energy sources?
The following corollary provides a solution for this ideal situation.
\begin{corollary}\label{cor1}
Assume that the conditions
     \begin{equation}\label{condi2}
          0<w(t)-r(t)\leq\min_{i\in\mathcal{N}}\left\{\frac{\nu_i^{\max}}{g_i}\right\}g_{\mathcal{N}},~~\forall 1\leq t\leq T,
     \end{equation}
and
\begin{equation}\label{newtotal}
     \sum_{t=1}^T w(t)=g_{\mathcal{N}}+\sum_{t=1}^T r(t)
\end{equation}
hold. Let
\begin{eqnarray*}
&&\bm{\pi}_0^*:=1+\frac{1}{n}+\left(\frac{\nu_{\mathcal{N}}(1)}{w(1)-r(1)},\ldots,\frac{\nu_{\mathcal{N}}(T)}{w(T)-r(T)}\right)^{\top},\\
&&\bm{\chi}_i^*:= \frac{g_i}{g_{\mathcal{N}}}\left(w(1)-r(1),\ldots, w(T)-r(T) \right)^{\top},~i\in\mathcal{N},\\
&&\bm{\chi}^*:=\left(\bm{\chi}_1^*,\bm{\chi}_2^*,\ldots,\bm{\chi}_n^*\right)^{\top}.
\end{eqnarray*}
Then, the Stackelberg game \eqref{leader}-\eqref{followers} has a perfect SE $(\bm{\pi}_0^*,\bm{\chi}^*)$, in which all electricity is provided by uncontrollable energy sources such as wind and photovoltaic power, without energy storage or waste.
\end{corollary}

% The condition \eqref{condi2} in Corollary \ref{cor1} is concise. Since $\nu_i^{\max} \geq g_i/T$, we have $$\min_{i\in\mathcal{N}}\left\{\frac{\nu_i^{\max}}{g_i}\right\}\geq\frac{1}{T}.$$ So, if $0<w(t)-r(t)\leq\frac{G}{T}, t=1,\ldots,T$ hold, the condition \eqref{condi2} is satisfied.
Fig.~\ref{Fig3} shows an example of the curves of $\{w(t),r(t)\}_{t=1,\ldots,T}$ and $\{\nu_i^*(t)\}_{i=1,\ldots,n;t=1,\ldots,T}$, which satisfies the conditions in Corollary \ref{cor1}. Therefore, it also presents a perfect SE profile, where all electricity is provided by uncontrollable energy sources.
% the number of time slots in a day as $T=24$, with each time slot being $1$ hour.
\begin{figure}[htp]
     \centering
     \includegraphics[width=3.2in]{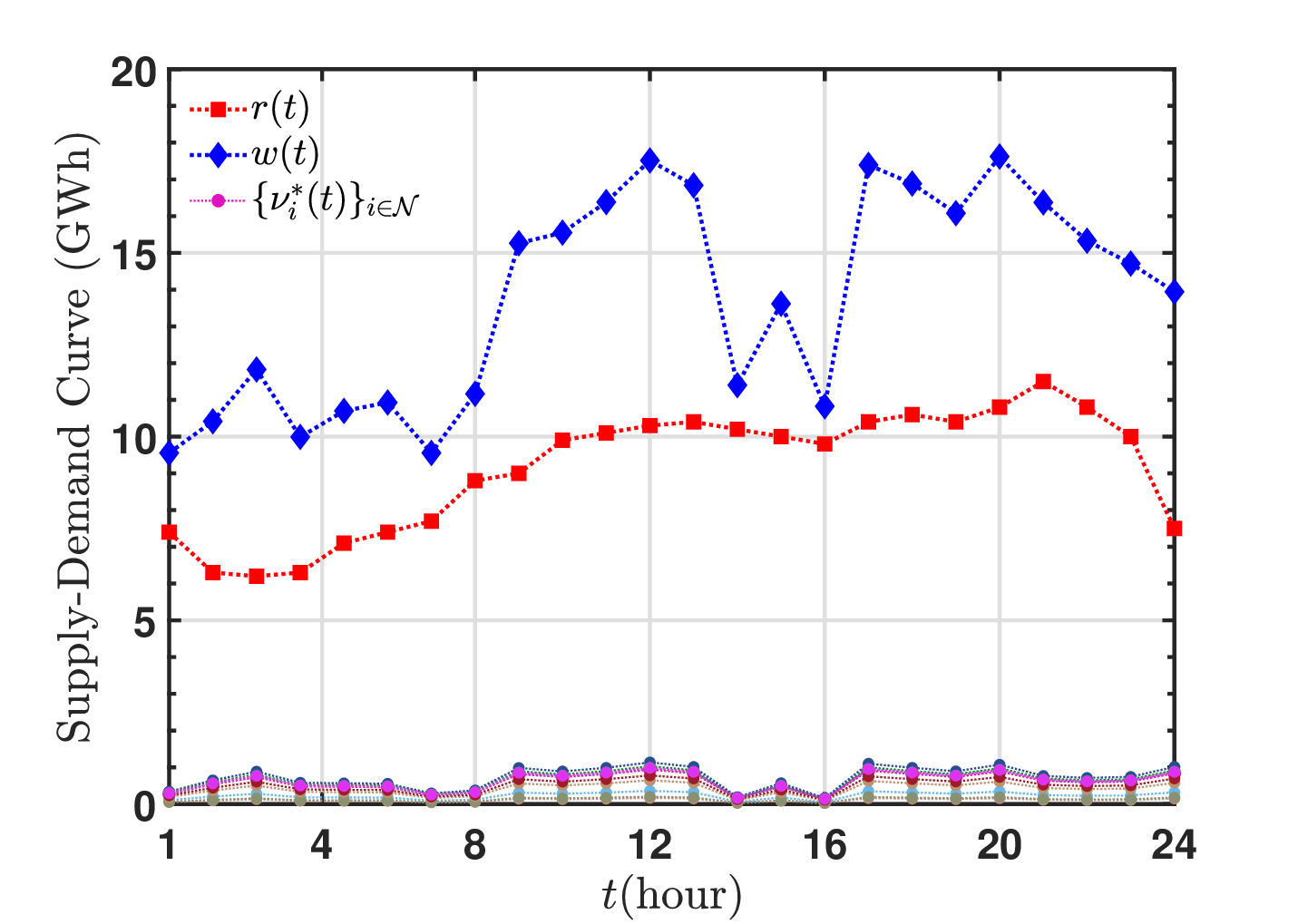}
     \caption{An ideal situation of perfect SE,  in which the uncontrollable energy accounts for 100\% of energy supply, and flexible users account for 33.64\% of the total energy demand.}
     \label{Fig3}
\end{figure}

\begin{remark}
One possible application of  Corollary \ref{cor1} is to plan the electricity consumption and power of flexible electricity users.
For example,  to absorb excess new energy generation we need build a hydrogen production plant. Using  Corollary \ref{cor1}, we can get
the required amount of hydrogen production power and capacity to achieve complete consumption of new energy.
\end{remark}

%在价格公式5中，平均项不能被直接得到，因此我们估计它为c
The price formula \eqref{priceform} has a defect that the item $\frac{1}{T}\sum_{s=1}^T[r(s)-w(s)]$ requires the future information of $r(t)$ and $w(t)$.
To avoid this defect we use the prediction value $b$ instead of $\frac{1}{T}\sum_{s=1}^T[r(s)-w(s)]$, and let
\begin{align}\label{priceformb}
\pi(b,t):=1+\frac{1}{n}+\frac{\nu_{\mathcal{N}}(t)}{w(t)+\dfrac{g_{\mathcal{N}}}{T}-r(t)+b},~~\forall 1\leq t\leq T.
\end{align}
Denote
\begin{equation}\label{preerr}
     \delta=\delta(b):=b-\frac{1}{T}\sum_{s=1}^T[r(s)-w(s)]
\end{equation}
as the prediction error.
If the utility company adopts the price formula $\pi(b,t)$,
a natural problem is that how much does the power generation cost of utility company increase compared to the perfect SE in Theorem \ref{thm1}? The following theorem provides a precise result.
\begin{theorem}\label{estcor}
     Let
     \begin{eqnarray*}
     \begin{aligned}
     \bm{\pi}_0&:=(\pi(b,1),\pi(b,2),\ldots,\pi(b,T))^{\top},\\
     \wideparen{w}(t)&:=w(t)-r(t)+\frac{g_{\mathcal{N}}}{T}+b,~ t=1,\ldots,T,\\
     \bm{\sigma}_i&:=\frac{g_i}{g_{\mathcal{N}}+T\delta}\left(\wideparen{w}(1),\ldots,\wideparen{w}(T)\right)^{\top},~i=1,\ldots,n,\\
     \bm{\sigma}&:=\left(\bm{\sigma}_1,\bm{\sigma}_2,\ldots,\bm{\sigma}_n\right)^{\top},
     \end{aligned}
     \end{eqnarray*}
     and assume that
     \begin{multline}\label{condi3}
          0<\wideparen{w}(t)\leq (g_{\mathcal{N}}+T\delta) \min_{i\in\mathcal{N}}\left\{\frac{\nu_i^{\max}}{g_i}\right\},
          ~~\forall 1\leq t\leq T.
     \end{multline}
%     where $\wideparen{w}(t)=w(t)-r(t)+\frac{g_{\mathcal{N}}}{T}+b, t=1\ldots,T$.
     Then, $\bm{\sigma}$ is a strict and unique Nash equilibrium of model \eqref{followers} when
     the leader's strategy is $\bm{\pi}_0$, while the cost function of the leader is
     \begin{equation}\label{restcor2}
          u^l(\bm{\pi}_0,\bm{\sigma})=\left(\frac{T\delta}{g_{\mathcal{N}}+T\delta}\right)^2\mbox{Var}\left(\{w(t)-r(t)\}_{t=1}^T\right),
     \end{equation}
     where
     \begin{multline*}
          \mbox{Var}\left(\{w(t)-r(t)\}_{t=1}^T\right):=\\
          \frac{1}{T}\sum_{t=1}^T\bigg(w(t)-r(t)-\frac{1}{T}\sum_{s=1}^T[w(s)-r(s)]\bigg)^2
     \end{multline*}
     denotes the variance of $\{w(t)-r(t)\}_{t=1}^T$.
\end{theorem}

The proofs of Theorem \ref{thm1}, Corollary \ref{cor1} and Theorem \ref{estcor} are put in Subsection \ref{PTheorem1}.

\begin{figure}[htp]
     \centering
     \includegraphics[width=3.48in]{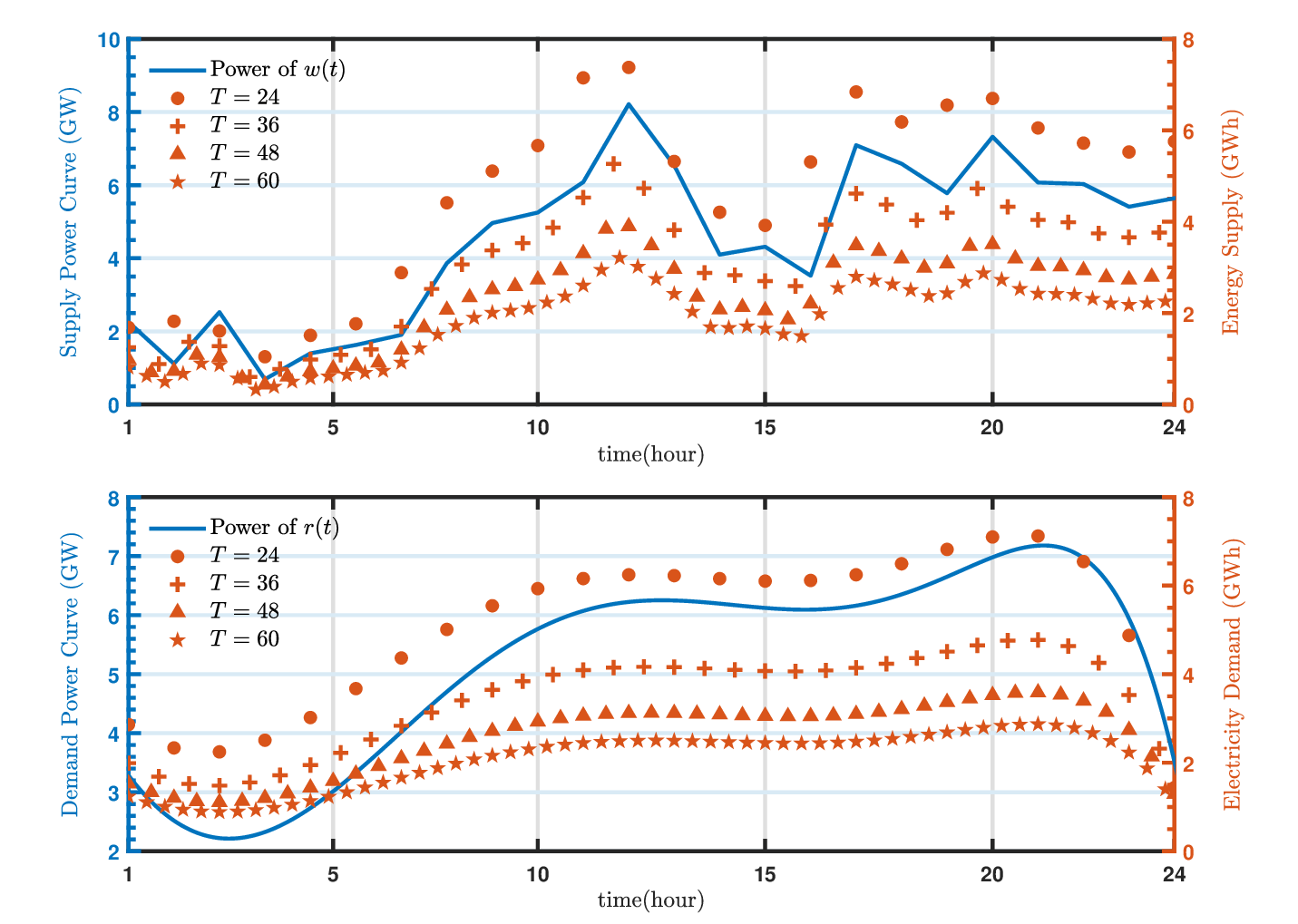}
     \caption{A numerical case of the new energy generation  (top) and the regular users' electricity consumption (bottom) in $24$ hours.
  The blue curves denote the power curves of new energy generation (top) and regular users' electricity consumption (bottom), while
       each scatter plot ($\bullet$ {\small\bf +} $\blacktriangle$ $\star$) denotes the  new energy generation (top) or regular users’ electricity consumption (bottom) at each sampling slot under different sampling numbers $T$.}
     \label{Fig6}
\end{figure}
% \begin{figure}[htp]
% 	\centering
% 	\subfigure[]{
% 	  \includegraphics[width=2.8in]{Figure/thm6w}
% 	}
% 	\subfigure[]{
% 	  \includegraphics[width=2.8in]{Figure/thm6r}
% 	}
% 	\caption{A scenario of the supply-demand power curve and the scatter plots of uncontrollable
%      energy supply and regular users' energy demand under different number of scheduling time window $T$.}
% 	\label{Fig6}
% \end{figure}

From Theorem \ref{estcor}, an interesting conclusion can be deduced that under the same prediction accuracy of new energy generation and regular user electricity consumption, the utility company's cost function $u^l(\bm{\pi}_0,\bm{\sigma})$ under the Nash equilibrium can be reduced greatly by shortening the length of the time slot, or increasing the sampling number $T$. In fact, when the prediction value of $\sum_s [r(s)-w(s)]$ in a certain time window (such as $24$ hours) remains unchanged,
if we increase the sampling number from $T$ to $2T$, then the prediction error $\delta$ will decrease to half by \eqref{preerr}.
On the other hand, we recall that $w(t)$ and $r(t)$ denote the new energy generation and regular user electricity consumption at time slot $t$ respectively,
so if we decrease the length of each time slot to half of the original, the values of $w(\cdot)$ and $r(\cdot)$ will also be down by nearly half at corresponding times.
By (\ref{restcor2}), the leader's cost function  $u^l(\bm{\pi}_0,\bm{\sigma})$ will
reduce to nearly $1/4$ of the original.

To verify this conclusion we make some simulations. We predict that the $24$-hour  new energy generation  and  regular users' electricity consumption are $125$GWh and $120$GWh respectively in a certain area.
However, their real data are $110.1$GWh and  $121.1$GWh respectively, whose power distributions are shown in Fig.~\ref{Fig6}.
%shows a numerical case of the new energy generation  and the regular users' electricity consumption  in $24$ hours.
%In this case the total new energy generation is $110.1$(GWh), and the total regular user electricity consumption is $121.1$(GWh).
%We predict
% its prediction value is $115.3$.
%Meanwhile, the total regular user electricity consumption is $121.1$, and its prediction value is $131.1$.
%Fig.~\ref{Fig6} shows the power curves of new energy generation and regular users' electricity consumption. Based on this, the scatter plots of new energy supply and regular users’ electricity consumption $\{w(t),r(t)\}_{t=1,\ldots,T}$ at the sampling number $T=24, 36, 48, 60$ are also obtained.
Then, by \eqref{preerr} the difference between the prediction value and the actual value of $\sum_s [r(s)-w(s)]$ is
\begin{equation*}
     \begin{aligned}
          T\delta &=Tb-\sum_s [r(s)-w(s)]\\
                  &=(120-125)-(121.1-110.1)=-16\mbox{GWh}.
     \end{aligned}
\end{equation*}
Besides, the total demand of flexible electricity users is
$g_{\mathcal{N}}=41.6$GWh in $24$ hours.
% Assume that the power curves of uncontrollable energy supply
 With above data, we calculate the prediction error $\delta$, the variance of $\{w(t)-r(t)\}_{t=1}^T$, and the cost function $u^l(\bm{\pi}_0,\bm{\sigma})$ of the utility company for different sampling numbers $T$, as shown in Table \ref{costfuncT}.
% Table \ref{costfuncT} calculates the cost function $u^l(\bm{\pi}_0,\bm{\sigma})$ of the leader under different scheduling time window $T$ , where we set $g_{\mathcal{N}}=41.6$.
According to Table \ref{costfuncT}, as the sampling number $T$ increases, the cost function of utility company gradually decreases.

\begin{table}[ht]
     \renewcommand{\arraystretch}{1.5}
     \caption{The cost function $u^l(\bm{\pi}_0,\bm{\sigma})$ for different $T$.}
     \centering
	\begin{tabular}{lcccc}
		\toprule
		$T$ & $24$ & $36$ & $48$ & $60$\\
		\midrule
          $\delta(\mbox{GWh})$ & $-2/3$ & $-4/9$ & $-1/3$ & $-4/15$\\
          $\mbox{Var}(\{w(t)-r(t)\}_{t=1}^T)(\mbox{GWh}^2)$ & $1.74$ & $0.84$ & $0.48$ & $0.31$\\
          $u^l(\bm{\pi}_0,\bm{\sigma})(\mbox{GWh}^2)$ & $0.68$ & $0.33$ & $0.19$ & $0.12$\\
		\bottomrule
	\end{tabular}
     \label{costfuncT}
\end{table}

\renewcommand{\thesection}{\Roman{section}}
\section{Proofs of Main Results\label{sec4}}
\renewcommand{\thesection}{\arabic{section}}

In this section, we provide the proofs of the corresponding theoretical results in Section \ref{sec3}.

To simplify the exposition, for any  $t=1,\ldots,T$ and $i=1,\ldots,n$ we set
\begin{eqnarray}\label{symb1}
\begin{aligned}
% &\widetilde{w}^*(t):=w(t)+a_1^*(t),\\
% &\widetilde{r}^*(t):=r(t)+a_2^*(t),\\
&\bm{d}_i:=(\nu_i(1),\nu_i(2),\ldots,\nu_i(T-1))^{\top},\\
&h_i(t):=\frac{1}{\widetilde{w}(t)}\big[\sum_{j\neq i}\nu_j(t)+\widetilde{r}(t)\big],
\end{aligned}
\end{eqnarray}
and
\begin{equation}\label{symb2}
     \begin{aligned}
          \bm{\mu}_i:&=
     \begin{pmatrix}
          h_i(1)-h_i(T)-\dfrac{2g_i}{\widetilde{w}(T)}\\
          h_i(2)-h_i(T)-\dfrac{2g_i}{\widetilde{w}(T)}\\
          \vdots\\
          h_i(T-1)-h_i(T)-\dfrac{2g_i}{\widetilde{w}(T)}\\
     \end{pmatrix},\\
     \mathbf{C}:&=
     \begin{pmatrix}
          \frac{1}{\widetilde{w}(1)} & 0 & \cdots & 0\\
          0 & \frac{1}{\widetilde{w}(2)} & \cdots & 0\\
          \vdots & \vdots & \ddots & \vdots\\
          0 & 0 & \cdots & \frac{1}{\widetilde{w}(T-1)}\\
     \end{pmatrix}+\frac{1}{\widetilde{w}(T)}\mathbf{1}\mathbf{1}^{\top}.
     % \begin{pmatrix}
     %      1 & \cdots & 1\\
     %      \vdots & \ddots & \vdots\\
     %      1 & \cdots & 1\\
     % \end{pmatrix}.
     \end{aligned}
\end{equation}
In the following discussion, the definitions of variables $\bm{\bar{d}}_i$ and $\bm{d}_i'$ are the same as $\bm{d}_i$ in \eqref{symb1} after replacing $\nu_i(t)$ with $\bar{\nu}_i(t)$ and $\nu_i'(t)$.
% The definition of variable $h_i^*(t)$ is the same as $h_i(t)$ in \eqref{symb1} after replacing $\widetilde{w}(t)$ and $\widetilde{r}(t)$ with $\widetilde{w}^*(t)$ and $\widetilde{r}^*(t)$, the symbols $\bm{\mu}_i^*$ and $\mathbf{C}^*$ are also similar.

\subsection{Proofs of Propositions \ref{saddle}-\ref{leaderop}}\label{PTheorem0}

One important factor why we chose \eqref{priceform0} as the leader's strategy is that under this strategy, the objective function of each follower is a strictly convex quadratic function, which can be formulated as the following lemma:
\begin{lemma}\label{bianxing}
Assume that the leader's strategy adopts \eqref{priceform0}, and each follower $i$'s  strategy $\bm{\nu}_i$ satisfies $\bm{1}^{\top} \bm{\nu}_i=g_i$. Then,
     the objective function $u_i^f(\bm{\pi},\bm{\nu}_i,\bm{\nu}_{-i})$ of each follower $i$ in model (\ref{followers}) can be written as
     \begin{equation}\label{quafun}
          u_i^f(\bm{\pi},\bm{\nu}_i,\bm{\nu}_{-i})=\bm{d}_i^\top \mathbf{C}\bm{d}_i+\bm{\mu}_i^{\top}\bm{d}_i+h_i(T)g_i+\frac{g_i^2}{\widetilde{w}(T)}.
     \end{equation}
In addition, if  $\widetilde{w}(t)>0$ for any $1\leq t\leq T$, then $u_i^f(\bm{\pi},\bm{\nu}_i,\bm{\nu}_{-i})$ is a strictly convex quadratic function.
\end{lemma}

The proof of Lemma \ref{bianxing} is postponed to Appendix \ref{appA}.

According to Lemma \ref{bianxing} and the convex analysis method \cite{borwein2006convex}, the objective function of each follower $i$ has a unique global minimum point $\bm{\bar{d}}_i$ satisfying
\begin{equation}\label{eq7_1}
     2\mathbf{C}\bm{\bar{d}}_i+\bm{\mu}_i=0.
\end{equation}
We get that \eqref{saddpoi} is the unique solution of linear equation system formed by uniting \eqref{eq7_1} of all followers.

\begin{lemma}\label{les}
If $\widetilde{w}(t)>0$ for any $1\leq t\leq T$, then the strategy combination \eqref{saddpoi} is the unique solution of linear equation system
\begin{equation*}
  2\mathbf{C}\bm{\bar{d}}_i+\bm{\mu}_i=0,~~i=1,2,\ldots,n.
\end{equation*}
\end{lemma}

The proof of Lemma \ref{les} is postponed to Appendix \ref{appB}.

To prove the uniqueness of the Nash equilibrium in game model \eqref{follower2} we need the well-known Banach's fixed point theorem:

\begin{lemma}[Banach's Fixed Point Theorem \cite{kantorovich2016functional}]\label{Banach}
     Let $(X,\mbox{dist})$ be a complete metric space and let $\tilde{F}: X\rightarrow X$ be a contraction mapping on $X$, that is, there exists a positive constant $\kappa<1$ such that $$\mbox{dist}(\tilde{F}(x),\tilde{F}(y))\leq\kappa \mbox{dist}(x,y), \forall x, y\in X.$$
     Then $\tilde{F}$ has a unique fixed point $x\in X$ (such that $\tilde{F}(x)=x$).
\end{lemma}

Now we construct a  contraction mapping for the game model \eqref{follower2}.  For each user $i$, we define its strategy hyperplane by $\mathbb{H}_i:=\{\bm{x}\in\mathbb{R}^T: \bm{1}^{\top} \bm{x}=g_i\}$. Set $\mathbb{H}_{-i}:=\mathbb{H}_1\times\cdots\times \mathbb{H}_{i-1}\times\mathbb{H}_{i+1}\times\cdots\times \mathbb{H}_{n}$ and $\mathbb{H}:=\mathbb{H}_1\times\cdots\times\mathbb{H}_{n}$. By Lemma \ref{bianxing} and \eqref{eq7_1}, we know that for each user $i$ and any strategy combination $\bm{\nu}_{-i}\in\mathbb{H}_{-i}$ of other users, the best response strategy of user $i$ in $\mathbb{H}_i$ is unique.
Thus, for any $\bm{\nu}\in\mathbb{H}$, we set the ``sequential'' best response mapping $\bm{F}=(\bm{f}_1,\bm{f}_2,\ldots,\bm{f}_n)^{\top}: \mathbb{H}\rightarrow\mathbb{H}$ with
\begin{eqnarray}\label{seqbs}
&&\bm{f}_i(\bm{\nu}):=\\
&&\mathop{\arg\min}_{\bm{\nu}_i\in\mathbb{H}_i}u_i^f\big(\bm{\pi}, \bm{\nu}_{i},(\bm{f}_1(\bm{\nu}),\ldots,\bm{f}_{i-1}(\bm{\nu}),\bm{\nu}_{i+1},\ldots,\bm{\nu}_{n})^{\top}\big)\nonumber
\end{eqnarray}
for  $i=1,2,\ldots,n$.
We show that:
\begin{lemma}\label{conmap}
$\bm{F}$ is a contraction mapping on  $\mathbb{H}$.
\end{lemma}

The proof of Lemma \ref{conmap} is postponed to Appendix \ref{appC}.

\emph{Proof of Proposition \ref{saddle}:}
For any user $i\in\mathcal{N}$ and any strategy $\bm{\nu}_i'=(\nu_i'(1),\ldots,\nu_i'(T))^{\top}\ne\bm{\nu}_i$ with $\bm{1}^{\top} \bm{\nu}_i'=g_i$, according to \eqref{quafun} we have
\begin{equation*}\label{strictNE}
     \begin{aligned}
          &u_i^f(\bm{\pi},\bm{\nu}_i',\bm{\nu}_{-i})-u_i^f(\bm{\pi},\bm{\nu}_{i},\bm{\nu}_{-i})\\
          &=\bm{d}_i'^\top \mathbf{C}\bm{d}_i'+\bm{\mu}_i^{\top}\bm{d}_i'+h_i(T)g_i+\frac{g_i^2}{\widetilde{w}(T)}\\
          &\quad-\Big(\bm{\bar{d}}_i^{\top} \mathbf{C}\bm{\bar{d}}_i+\bm{\mu}_i^{\top}\bm{\bar{d}}_i+h_i(T)g_i+\frac{g_i^2}{\widetilde{w}(T)}\Big)\\
          &=(\bm{d}_i'-\bm{\bar{d}}_i)^{\top}\mathbf{C}(\bm{d}_i'-\bm{\bar{d}}_i)+\bm{d}_i'^{\top}\mathbf{C}\bm{\bar{d}}_i+\bm{\bar{d}}_i^{\top}\mathbf{C}\bm{d}_i'\\
          &\quad-2\bm{\bar{d}}_i^{\top}\mathbf{C}\bm{\bar{d}}_i+(\bm{d}_i'-\bm{\bar{d}}_i)^{\top}\bm{\mu}_i\\
          &=(\bm{d}_i'-\bm{\bar{d}}_i)^{\top}\mathbf{C}(\bm{d}_i'-\bm{\bar{d}}_i)+2(\bm{d}_i'-\bm{\bar{d}}_i)^{\top}\mathbf{C}\bm{\bar{d}}_i\\
          &\quad+(\bm{d}_i'-\bm{\bar{d}}_i)^{\top}\bm{\mu}_i\\
          &=(\bm{d}_i'-\bm{\bar{d}}_i)^{\top}\mathbf{C}(\bm{d}_i'-\bm{\bar{d}}_i)+(\bm{d}_i'-\bm{\bar{d}}_i)^{\top}(2\mathbf{C}\bm{\bar{d}}_i+\bm{\mu}_i)\\
          &=(\bm{d}_i'-\bm{\bar{d}}_i)^{\top}\mathbf{C}(\bm{d}_i'-\bm{\bar{d}}_i)>0,
     \end{aligned}
\end{equation*}
where the third equality and the last inequality use the fact that  $\mathbf{C}$ is symmetric positive definite, and the last equality uses \eqref{eq7_1}.
% Since $\bm{d}_i'-\bm{d}_i\ne 0$ and $\mathbf{C}$ is symmetric positive definite, we have $$u_i^f(\bm{\pi},\bm{\nu}_i',\bm{\nu}_{-i})>u_i^f(\bm{\pi},\bm{\nu}_{i},\bm{\nu}_{-i}),$$
So $\bm{\nu}$ is a strict Nash equilibrium of model \eqref{follower2} when the leader's strategy is $\bm{\pi}$.

It remains to prove the uniqueness of the Nash equilibrium of model \eqref{follower2}. Assume that $\bm{\widetilde{\nu}}\in\mathbb{H}$ is  an arbitrary  Nash equilibria, then by \eqref{NEdef} we have
\begin{equation}\label{NEeq}
     \bm{\widetilde{\nu}}_i=\mathop{\arg\min}_{\bm{\nu}_i\in\mathbb{H}_i}u_i^f(\bm{\pi},\bm{\nu}_i,\bm{\widetilde{\nu}}_{-i}), ~~\forall i\in\mathcal{N}.
\end{equation}
Here we recall that the best response strategy of user $i$ in $\mathbb{H}_i$ is unique. By \eqref{seqbs} and \eqref{NEeq}  we can sequentially get
\begin{equation*}
     \bm{f}_1(\bm{\widetilde{\nu}})=\mathop{\arg\min}_{\bm{\nu}_1\in\mathbb{H}_1}u_1^f\big(\bm{\pi},\bm{\nu}_1,(\bm{\widetilde{\nu}}_{2},\ldots,\bm{\widetilde{\nu}}_{n})^{\top}\big)=\bm{\widetilde{\nu}}_1
\end{equation*}
and
\begin{equation*}
     \begin{aligned}
          &\bm{f}_i(\bm{\widetilde{\nu}})=\\
          &\mathop{\arg\min}_{\bm{\nu}_i\in\mathbb{H}_i}u_i^f\big(\bm{\pi}, \bm{\nu}_{i},(\bm{f}_1(\bm{\widetilde{\nu}}),\ldots,\bm{f}_{i-1}(\bm{\widetilde{\nu}}),\bm{\widetilde{\nu}}_{i+1},\ldots,\bm{\widetilde{\nu}}_{n})^{\top}\big)\\
          &=\mathop{\arg\min}_{\bm{\nu}_i\in\mathbb{H}_i}u_i^f\big(\bm{\pi}, \bm{\nu}_{i},(\bm{\widetilde{\nu}}_1,\ldots,\bm{\widetilde{\nu}}_{i-1},\bm{\widetilde{\nu}}_{i+1},\ldots,\bm{\widetilde{\nu}}_{n})^{\top}\big)\\
          &=\bm{\widetilde{\nu}}_i, \qquad\qquad\qquad\qquad\qquad\qquad\qquad i=2,\ldots,n.
     \end{aligned}
\end{equation*}
So we have $\bm{F}(\bm{\widetilde{\nu}})=\bm{\widetilde{\nu}}$, that is, $\bm{\widetilde{\nu}}$ is a fixed point of $\bm{F}$.

% Lastly, we prove that $\bm{\nu}$ is the unique Nash equilibrium of model \eqref{follower2}.
Therefore, all Nash equilibria of model \eqref{follower2} are fixed points of the mapping $\bm{F}$.
Based on Lemmas \ref{Banach}-\ref{conmap}, $\bm{F}$ has a unique fixed point, thus model \eqref{follower2} has a unique Nash equilibrium.
Since $\bm{\nu}\in\mathbb{H}$ is a strict Nash equilibrium for model \eqref{follower2}, $\bm{\nu}$ must be the unique Nash equilibrium of model \eqref{follower2}.
% For each follower $i$, we know that $u_i^f(\bm{\pi},\bm{\nu}_i,\bm{\nu}_{-i})$ is a strictly convex quadratic function. When $\|\bm{d}_i\|_2\rightarrow\infty$, $u_i^f(\bm{\pi},\bm{\nu}_i,\bm{\nu}_{-i})\rightarrow\infty$. To minimize the cost function, we can consider the strategy $\bm{d}_i$ to be within a bounded closed convex region.
\qed

\vskip 2mm

Set $\bm{d}:=(\bm{d}_1^{\top},\ldots,\bm{d}_n^{\top})^{\top}$ to be the strategy vector of all followers, and $$\bm{g}(\bm{d}):=\left(\frac{\partial u_1^f(\bm{\pi},\bm{\nu}_1,\bm{\nu}_{-1})}{\partial\bm{d}_1^{\top}},\ldots,\frac{\partial u_n^f(\bm{\pi},\bm{\nu}_n,\bm{\nu}_{-n})}{\partial\bm{d}_n^{\top}}\right)^{\top}.$$
Further, we make $\mathbf{J}(\bm{d})$ the Jacobian with respect to $\bm{d}$ of $\bm{g}(\bm{d})$, i.e.,
\begin{equation*}
     \mathbf{J}(\bm{d}):=
     \begin{pmatrix}
          \frac{\partial^2 u_1^f(\bm{\pi},\bm{\nu}_1,\bm{\nu}_{-1})}{\partial\bm{d}_1^2} & \cdots & \frac{\partial^2 u_1^f(\bm{\pi},\bm{\nu}_1,\bm{\nu}_{-1})}{\partial\bm{d}_1\partial\bm{d}_n}\\
          \vdots & \ddots & \vdots\\
          \frac{\partial^2 u_n^f(\bm{\pi},\bm{\nu}_n,\bm{\nu}_{-n})}{\partial\bm{d}_n\partial\bm{d}_1} & \cdots & \frac{\partial^2 u_n^f(\bm{\pi},\bm{\nu}_n,\bm{\nu}_{-n})}{\partial\bm{d}_n^2}\\
     \end{pmatrix}.
\end{equation*}
In the model \eqref{followers}, for each user $i\in\mathcal{N}$, we define the strategy space $\Omega_i:=\mathbb{H}_i\cap [0,\nu_i^{\max}]^T$ as the set of its feasible strategies, which is a bounded closed convex set.
According to Theorem 2 in \cite{rosen1965existence}, we can get that the Nash equilibrium of model \eqref{followers} is unique, which can be formulated as the following lemma:
\begin{lemma}[Theorem 2 in \cite{rosen1965existence}]\label{uniNE}
     For each user $i\in\mathcal{N}$, if $u_i^f(\bm{\pi},\bm{\nu}_i,\bm{\nu}_{-i})$ is a convex function on the bounded closed convex region $\Omega_i$,
     % assume that $u_i^f(\bm{\pi},\bm{\nu}_i,\bm{\nu}_{-i})$ is a convex function  and the strategy space $\Omega_i$ is a bounded closed convex set.
    and $\mathbf{J}(\bm{d})$ is symmetric positive definite for all $\bm{d}$,
     then model \eqref{followers} admits a unique Nash equilibrium.
\end{lemma}

\emph{Proof of Proposition \ref{oristruni}:}
For any $1\leq i\leq n$, according to Lemma \ref{bianxing}, $u_i^f(\bm{\pi},\bm{\nu}_i,\bm{\nu}_{-i})$ is a strictly convex quadratic function.
By \eqref{eq7} and \eqref{muchange}, we have
\begin{equation*}
          \bm{g}(\bm{d})=
     \begin{pmatrix}
          \mathbf{C}\Big[2\bm{d}_1+\sum_{j\neq 1}\bm{d}_j\Big]-\widetilde{\bm{g}}_1\\
          \vdots\\
          \mathbf{C}\Big[2\bm{d}_n+\sum_{j\neq n}\bm{d}_j\Big]-\widetilde{\bm{g}}_n\\
     \end{pmatrix}
\end{equation*}
and
\begin{equation*}
     \begin{aligned}
          \mathbf{J}(\bm{d})&=
          \begin{pmatrix}
               2\mathbf{C} & \cdots & \mathbf{C}\\
               \vdots & \ddots & \vdots\\
               \mathbf{C} & \cdots & 2\mathbf{C}\\
          \end{pmatrix}=
          \begin{pmatrix}
               2 & \cdots & 1\\
               \vdots & \ddots & \vdots\\
               1 & \cdots & 2\\
          \end{pmatrix}\otimes\mathbf{C},
     \end{aligned}
\end{equation*}
where ``$\otimes$'' denotes the Kronecker product. Since the Kronecker product of two positive definite matrices is still positive definite \cite{laub2004matrix}, we can get that $\mathbf{J}(\bm{d})$ is symmetric positive definite for all $\bm{d}$.
According to Lemma \ref{uniNE}, we can conclude that model \eqref{followers} admits a unique Nash equilibrium denoted as  $\bm{\nu}$.

Now we show that $\bm{\nu}$ is a strict Nash equilibrium of model \eqref{followers}.
% where $\mathbb{H}_i:=\{\bm{x}\in\mathbb{R}^T: \bm{1}^{\top} \bm{x}=g_i\}$.
For any user $i$ and any strategy $\bm{\nu}_i'=(\nu_i'(1),\ldots,\nu_i'(T))^{\top}\in\Omega_i$ with $\bm{\nu}_i'\ne\bm{\nu}_i$, by \eqref{NEdef} we have $$u_i^f(\bm{\pi},\bm{\nu}_i',\bm{\nu}_{-i})\geq u_i^f(\bm{\pi},\bm{\nu}_{i},\bm{\nu}_{-i}).$$
It is clear that $\frac{\bm{\nu}_i'+\bm{\nu}_i}{2}\in\Omega_i$, so
\begin{equation}\label{strconvex}
     u_i^f\big(\bm{\pi},\frac{\bm{\nu}_i'+\bm{\nu}_i}{2},\bm{\nu}_{-i}\big)\geq u_i^f(\bm{\pi},\bm{\nu}_{i},\bm{\nu}_{-i}).
\end{equation}
If $u_i^f(\bm{\pi},\bm{\nu}_i',\bm{\nu}_{-i})=u_i^f(\bm{\pi},\bm{\nu}_{i},\bm{\nu}_{-i})$, then according to strict convexity of $u_i^f(\bm{\pi},\bm{\nu}_{i},\bm{\nu}_{-i})$, we can get
\begin{equation*}
     \begin{aligned}
               u_i^f\big(\bm{\pi},\frac{\bm{\nu}_i'+\bm{\nu}_i}{2},\bm{\nu}_{-i}\big)&<\frac{u_i^f(\bm{\pi},\bm{\nu}_i',\bm{\nu}_{-i})+u_i^f(\bm{\pi},\bm{\nu}_{i},\bm{\nu}_{-i})}{2}\\
                                                                             &=u_i^f(\bm{\pi},\bm{\nu}_i,\bm{\nu}_{-i}),
     \end{aligned}
\end{equation*}
which contradicts \eqref{strconvex}. So $$u_i^f(\bm{\pi},\bm{\nu}_i',\bm{\nu}_{-i})>u_i^f(\bm{\pi},\bm{\nu}_{i},\bm{\nu}_{-i}),$$ which implies $\bm{\nu}$ is a strict Nash equilibrium of model \eqref{followers}.
\qed

\vskip 2mm

\emph{Proof of Proposition \ref{leaderop}:}
% For any $1\leq t\leq T$, when we choose $a_1^*(t)$ and $a_2^*(t)$ in \eqref{cstar}, by \eqref{symb1} we have
% \begin{equation}\label{thm1_p1}
%      \begin{aligned}
%           \widetilde{w}^*(t)&=\frac{g_{\mathcal{N}}}{T}+w(t)-r(t)+\frac{1}{T}\sum_{s=1}^T[r(s)-w(s)],\\
%           \widetilde{r}^*(t)&=\frac{n+1}{n}\widetilde{w}^*(t).
%      \end{aligned}
% \end{equation}
% Then, the Nash equilibrium in \eqref{saddpoi} becomes
% \begin{equation}\label{NE2}
%      \begin{aligned}
%           &\frac{\widetilde{w}^*(t)}{\sum_s\widetilde{w}^*(s)}g_i+\frac{1}{(n+1)\sum_s\widetilde{w}^*(s)}\\
%           &~~~~~~~~\quad\times\sum_{s=1\atop s\neq t}^T\widetilde{w}^*(s)\widetilde{w}^*(t)\left[\frac{\widetilde{r}^*(s)}{\widetilde{w}^*(s)}-\frac{\widetilde{r}^*(t)}{\widetilde{w}^*(t)}\right]\\
%           &=\frac{g_i}{g_{\mathcal{N}}}\widetilde{w}^*(t)=\nu_i^*(t),~~i=1,\ldots,n, t=1,\cdots,T,
%      \end{aligned}
% \end{equation}
% where the two equalities use \eqref{thm1_p1}.
Substituting \eqref{NE2} into the problem \eqref{leader}, we have the controllable energy generation $c^*(t)$ satisfies:
% When $c_1(t)=\frac{G}{T}-r(t)+\frac{1}{T}\sum_{s=1}^T[r(s)-w(s)]$ and $c_2(t)=\frac{n+1}{N}[w(t)+c_1(t)]-r(t)$, Eq.~\eqref{NE4} can be obtained, and we have
\begin{eqnarray*}\label{opt1}
          &&~~~~c^*(t)-\bar{c}\nonumber\\
          &&=\sum_{i\in\mathcal{N}}\nu_i^*(t)+r(t)-w(t)-\frac{g_{\mathcal{N}}}{T}-\frac{1}{T}\sum_{s=1}^T[r(s)-w(s)]\nonumber\\
          &&=\frac{g_{\mathcal{N}}}{T}+\frac{1}{T}\sum_{s=1}^T[r(s)-w(s)]-\frac{g_{\mathcal{N}}}{T}-\frac{1}{T}\sum_{s=1}^T[r(s)-w(s)]\nonumber\\
          &&=0, ~~~~\forall 1\leq t\leq T,
\end{eqnarray*}
which indicates the leader's cost function $u^l(\bm{\pi}^*,\bm{\nu}^*)=0$.
\qed

\subsection{Proofs of Theorem \ref{thm1}, Corollary \ref{cor1} and Theorem \ref{estcor}}\label{PTheorem1}

\emph{Proof of Theorem \ref{thm1}:}
     i)  Assume the leader's strategy is $\bm{\pi}^*$.  First, by Proposition \ref{saddle} and the discussions from \eqref{cstar} to  \eqref{NE2} we have $\bm{\nu}^*$ is a strict Nash equilibrium of model \eqref{follower2}. Also, by  \eqref{NE2} and \eqref{asum1} we have
     $$0<\nu_i^*(t) \leq g_i \min_{j\in\mathcal{N}}\left\{\frac{\nu_j^{\max}}{g_j}\right\}\leq g_i \frac{\nu_i^{\max}}{g_i}=\nu_i^{\max},$$
     which indicates that  $\nu_i^*(t)$ satisfies the bound constraint condition in model \eqref{followers}.
     Therefore, $\bm{\nu}^*$ is a strict Nash equilibrium of model \eqref{followers}. By Proposition \ref{oristruni},  $\bm{\nu}^*$ is the unique Nash equilibrium of model \eqref{followers}.

     ii) By i) and Proposition \ref{leaderop}, we have $(\bm{\pi}^*,\bm{\nu}^*)$ is a perfect SE of the Stackelberg game \eqref{leader}-\eqref{followers}.
\qed

\vskip 2mm

\emph{Proof of Corollary \ref{cor1}:}
     Under the condition \eqref{newtotal}, the price function \eqref{priceform} becomes
     \begin{equation*}
          \begin{aligned}
               \pi^*(t)&=1+\frac{1}{n}+\frac{\nu_{\mathcal{N}}(t)}{w(t)+\dfrac{g_{\mathcal{N}}}{T}-r(t)+\frac{1}{T}\sum_{s=1}^T[r(s)-w(s)]}\\
               &=1+\frac{1}{n}+\frac{\nu_{\mathcal{N}}(t)}{w(t)-r(t)}, ~~t=1,\ldots,T,
          \end{aligned}
     \end{equation*}
     and the Nash equilibrium \eqref{NE2} is
     \begin{equation*}
          \begin{aligned}
               \nu_i^*(t)&=\frac{g_i}{T}+\frac{g_i}{g_{\mathcal{N}}}[w(t)-r(t)]+\frac{g_i}{g_{\mathcal{N}} T}\sum_{s=1}^T[r(s)-w(s)]\\
               &=\frac{g_i}{g_{\mathcal{N}}}[w(t)-r(t)], ~~i=1,\ldots,n, t=1,\cdots,T.
          \end{aligned}
     \end{equation*}
     Therefore, according to the inequality \eqref{condi2} and Theorem \ref{thm1}, $\bm{\chi}^*$ is a strict and unique Nash equilibrium of model \eqref{followers} when the leader's strategy is $\bm{\pi}_0^*$, and $(\bm{\pi}_0^*,\bm{\chi}^*)$ is a perfect SE of the Stackelberg game \eqref{leader}-\eqref{followers}.
     %That is to say, all electricity is provided by uncontrollable energy sources without energy storage or waste.
%      By repeating Eq.~\eqref{NE2}-\eqref{strictNE} in the proof process of Theorem \ref{thm1}, we can obtain a strict Nash equilibrium of model \eqref{followers} as
%      \begin{equation*}
%           \begin{aligned}
%                \nu_i(t)&=\frac{\widetilde{w}(t)}{\sum_s\widetilde{w}(s)}g_i=\frac{\widetilde{w}(t)}{\sum_s[w(s)-r(s)]}g_i\\
%                      &=\frac{g_i}{g_{\mathcal{N}}}[w(t)-r(t)] \quad~~~\forall i\in\mathcal{N}, t=1,\cdots,T.
%           \end{aligned}
%      \end{equation*}
%      So, $\bm{\chi}^*$ is a strict Nash equilibrium of model \eqref{followers} when the leader's strategy is $\bm{\pi}_0^*$.
%      Further, for any $1\leq t\leq T$, the total electricity demand $\nu_{\mathcal{N}}(t)$ of all followers is
%      \begin{equation*}
%           \begin{aligned}
%                \nu_{\mathcal{N}}(t)&=\sum_{i\in\mathcal{N}}\nu_i(t)=\sum_{i\in\mathcal{N}}\frac{g_i}{g_{\mathcal{N}}}[w(t)-r(t)]=w(t)-r(t).
%           \end{aligned}
%      \end{equation*}
%      Then the controllable energy generation $c(t)=0$, $t=1,\ldots,T$, and $u^l(\bm{\pi}_0^*,\bm{\chi}^*)=0$, so $(\bm{\pi}_0^*,\bm{\chi}^*)$ is a perfect SE of the Stackelberg game \eqref{leader}-\eqref{followers}. That is to say, all electricity is provided by uncontrollable energy sources without energy storage or waste.
\qed

\vskip 2mm

\emph{Proof of Theorem \ref{estcor}:}
     When \eqref{priceformb} is chosen as the leader's strategy, by \eqref{priceform0} and \eqref{wandr}, we have $\widetilde{w}(t)=w(t)-r(t)+\frac{g_{\mathcal{N}}}{T}+b$ and $\widetilde{r}(t)=\frac{n+1}{n}\widetilde{w}(t)$.
     According to Proposition \ref{saddle}, the strict and unique Nash equilibrium in \eqref{saddpoi} becomes
     \begin{eqnarray}\label{NE3}
               && \nu_i(t)=\frac{\widetilde{w}(t)}{\sum_s\widetilde{w}(s)}g_i+\frac{1}{(n+1)\sum_s\widetilde{w}(s)}\nonumber\\
               &&~~\qquad\qquad\quad\times\sum_{s=1}^T\widetilde{w}(s)\widetilde{w}(t)\left[\frac{\widetilde{r}(s)}{\widetilde{w}(s)}-\frac{\widetilde{r}(t)}{\widetilde{w}(t)}\right]\nonumber\\
               &&=\frac{g_i}{g_{\mathcal{N}}+Tb+\sum_{s=1}^T[w(s)-r(s)]}\\
               &&~~\qquad\qquad\quad\times\big[w(t)-r(t)+\frac{g_{\mathcal{N}}}{T}+b\big]\nonumber\\
               &&=\frac{g_i}{g_{\mathcal{N}}+T\delta}\big[w(t)-r(t)+\frac{g_{\mathcal{N}}}{T}+b\big],\nonumber\\
               &&~~~~~\qquad\qquad\qquad~ i=1,\ldots,n,~ t=1,\ldots,T,\nonumber
     \end{eqnarray}
     where the last equality uses \eqref{preerr}.
     With the condition \eqref{condi3} and the similar process of the proof of Theorem \ref{thm1}, we can conclude that $\bm{\sigma}$ is a strict and unique Nash equilibrium of model \eqref{followers} when the leader's strategy is $\bm{\pi}_0$.

     % Under the conditions of Theorem \ref{estcor}, we have $\widetilde{w}(t)=w(t)-r(t)+\frac{g_{\mathcal{N}}}{T}+b$ and $\widetilde{r}(t)=\frac{n+1}{n}[w(t)-r(t)+\frac{g_{\mathcal{N}}}{T}+b]$.
     % By repeating Eq.~\eqref{NE2}-\eqref{strictNE} in the proof process of Theorem \ref{thm1}, we can obtain a strict Nash equilibrium for model \eqref{followers} as
     % \begin{equation*}
     %      \begin{aligned}
     %           &~~\nu_i(t)=\frac{\widetilde{w}(t)}{\sum_s\widetilde{w}(s)}g_i\\
     %           &=\frac{g_i}{g_{\mathcal{N}}+Tb+\sum_{s=1}^T[w(s)-r(s)]}\big[w(t)-r(t)+\frac{g_{\mathcal{N}}}{T}+b\big]\\
     %           &=\frac{g_i}{g_{\mathcal{N}}+T\delta}\big[w(t)-r(t)+\frac{g_{\mathcal{N}}}{T}+b\big],\\
     %           &\qquad\qquad\qquad\qquad\qquad\qquad\quad~~~\forall i\in\mathcal{N}, t=1,\cdots,T.
     %      \end{aligned}
     % \end{equation*}

     % Therefore, $\bm{\sigma}$ is a strict Nash equilibrium of model \eqref{followers} when
     % the leader's strategy is $\bm{\pi}_0$.

     Further, for any $1\leq t\leq T$,
     substituting \eqref{NE3} into the problem \eqref{leader}, we have the controllable energy generation $c(t)$ satisfies:
     \begin{equation*}
          \begin{aligned}
               &~~~~c(t)-\bar{c}\\
               &=\sum_{i\in\mathcal{N}}\nu_i(t)+r(t)-w(t)-\frac{g_{\mathcal{N}}}{T}-\frac{1}{T}\sum_{s=1}^T[r(s)-w(s)]\\
               &=\frac{g_{\mathcal{N}}}{g_{\mathcal{N}}+T\delta}\bigg(w(t)-r(t)+\frac{g_{\mathcal{N}}}{T}+\delta+\frac{1}{T}\sum_{s=1}^T[r(s)-w(s)]\bigg)\\
               &~~~+r(t)-w(t)-\frac{g_{\mathcal{N}}}{T}-\frac{1}{T}\sum_{s=1}^T[r(s)-w(s)]\\
               &=\bigg(\frac{g_{\mathcal{N}}}{g_{\mathcal{N}}+T\delta}-1\bigg)\frac{1}{T}\sum_{s=1}^T[r(s)-w(s)]+\bigg(\frac{g_{\mathcal{N}}}{g_{\mathcal{N}}+T\delta}-1\bigg)\\
               &~~~\times[w(t)-r(t)]\\
               &=\frac{\delta}{g_{\mathcal{N}}+T\delta}\sum_{s=1}^T[w(s)-r(s)]-\frac{T\delta}{g_{\mathcal{N}}+T\delta}[w(t)-r(t)],
          \end{aligned}
     \end{equation*}
     where the second equality uses \eqref{preerr}.
     From this, the cost function of the leader is
     \begin{equation*}
          \begin{aligned}
               &~~u^l(\bm{\pi}_0,\bm{\sigma})=\frac{1}{T}\sum_{t=1}^T\big[c(t)-\bar{c}\big]^2\\
               &=\frac{1}{T}\sum_{t=1}^T\bigg(\frac{\delta}{g_{\mathcal{N}}+T\delta}\sum_{s=1}^T[w(s)-r(s)]-\frac{T\delta}{g_{\mathcal{N}}+T\delta}\\
               &~~~\times[w(t)-r(t)]\bigg)^2\\
               &=\frac{(T\delta)^2}{(g_{\mathcal{N}}+T\delta)^2}\frac{1}{T}\sum_{t=1}^T\bigg(w(t)-r(t)-\frac{1}{T}\sum_{s=1}^T[w(s)-r(s)]\bigg)^2\\
               &=\left(\frac{T\delta}{g_{\mathcal{N}}+T\delta}\right)^2\mbox{Var}\left(\{w(t)-r(t)\}_{t=1}^T\right).
          \end{aligned}
     \end{equation*}
\qed

\renewcommand{\thesection}{\Roman{section}}
\section{Numerical Solution of SE for Model \eqref{leader}-\eqref{followers} without Condition  \eqref{asum1}\label{sec5}}
\renewcommand{\thesection}{\arabic{section}}

%线性定价方案形式仍然最优
%求解leader方的优化问题、设计(\widetilde{r}(t),\widetilde{w}(t))使得follower方的NE达到leader的最优解（算法1）、给出一个具体的实例（作图并说明）
In Subsection \ref{ssmainresult} we  give the analytic expression of SE for the Stackelberg game \eqref{leader}-\eqref{followers} based on the inequality condition \eqref{asum1}.
A natural problem is whether we can still obtain the SE if the condition \eqref{asum1} is not satisfied.
The analytic expression  of SE should be hard  for this case, however we can provide a  numerical solution.

%In this section, we demonstrate that when the supply and demand relationship does not meet the condition \eqref{asum1}, the numerical solutions of SE for the Stackelberg game \eqref{leader}-\eqref{followers} can be found. Moreover, we provide a specific simulation example.

First, let us review the optimization problem \eqref{leader} of the leader side.  In fact,  the leader optimizes its objective by adjusting the electricity demand of followers through price function.
If we omit the price function, the optimization problem \eqref{leader} can be rewritten as
\begin{eqnarray}\label{leader2}
	\begin{aligned}
		 &\min_{\{\nu_{\mathcal{N}}(t)\}_{t=1}^T}\frac{1}{T}\sum_{t=1}^T\big[c(t)-\bar{c}\big]^2,\\
		 &~~\mathrm{s.t.}~ \sum_{t=1}^T \nu_{\mathcal{N}}(t)=g_{\mathcal{N}}, \\
		 &~~~~~~~ \nu_{\mathcal{N}}(t)+r(t)=w(t)+c(t), \\
&~~~~~~~ 0\leq \nu_{\mathcal{N}}(t)\leq \nu_{\mathcal{N}}^{\max},~ t=1,\cdots,T,
	\end{aligned}
\end{eqnarray}
where $\nu_{\mathcal{N}}^{\max}:=\sum_{i\in\mathcal{N}}\nu_i^{\max}$.

The optimization problem \eqref{leader2} is a quadratic program with a strictly convex quadratic function. According to the convex analysis method \cite{borwein2006convex}, the problem \eqref{leader2} has a unique global optimal solution, which we denote as $\bm{\nu}_{\mathcal{N}}^*=(\nu_{\mathcal{N}}^*(1),\nu_{\mathcal{N}}^*(2),\cdots,\nu_{\mathcal{N}}^*(T))^{\top}$, and it can be numerically solved by several methods, such as the active-set method \cite{gould2005numerical}.

{
\begin{table}[ht]
\refstepcounter{table}
\hypertarget{algorithm1}{}
\begin{center}\footnotesize
\setlength{\tabcolsep}{6pt}
\begin{tabular}{p{8.5cm}}
\toprule
\textbf{Algorithm 1:} Best response algorithm for solving the game model \eqref{followers} \\
\midrule
\textbf{Input:}
The number of flexible users $n$, the number of time periods $T$, the value of $g_i$ and $\nu_i^{\max}$ for each user $i\in\mathcal{N}$, the pricing sequence $(\widetilde{r}(t),\widetilde{w}(t))_{1\leq t\leq T}$, the tolerance $\varepsilon_0$.
\\
\textbf{Initialization:}
Set $k=0$. Set the initial strategy for each flexible user $i\in\mathcal{N}$ as the average strategy $\bm{\nu}_i^{(0)}=(\frac{g_i}{T},\ldots,\frac{g_i}{T})^{\top}$.\\
\textbf{While not converged do}\\
\quad\textbf{1.} \textbf{for} $i=1,2,\cdots,n$ \textbf{do}\\
\quad\textbf{2.} \quad Calculate the best response strategy $\bm{\nu}_i^{(k+1)}$ for game model \eqref{followers}\\
\qquad\quad when fix $\bm{\nu}_{-i}=(\bm{\nu}_1^{(k+1)},\cdots,\bm{\nu}_{i-1}^{(k+1)},\bm{\nu}_{i+1}^{(k)},\cdots,\bm{\nu}_n^{(k)})^{\top}$.\\
\quad\textbf{3.} \textbf{end for}\\
\quad\textbf{4.} Calculate $er=\|\bm{\nu}^{(k+1)}-\bm{\nu}^{(k)}\|_1$.\\
\quad\textbf{5.} Update $k=k+1$.\\
\textbf{Until} $er<\varepsilon_0$\\
\textbf{Output:} The approximative Nash equilibrium $\bm{\nu}^{(k)}$.\\
\bottomrule
\end{tabular}
\end{center}
\end{table}
}

{
\begin{table}[ht]
\refstepcounter{table}
\hypertarget{algorithm2}{}
\begin{center}\footnotesize
\setlength{\tabcolsep}{6pt}
\begin{tabular}{p{8.5cm}}
\toprule
\textbf{Algorithm 2:} Search algorithm for the optimal pricing sequence \\
\midrule
\textbf{Input:}
The number of flexible users $n$, the number of time periods $T$, the value of $g_i$ and $\nu_i^{\max}$ for each user $i\in\mathcal{N}$, the regular user demand and new energy supply sequence $(r(t),w(t))_{1\leq t\leq T}$, the tolerance $\varepsilon_0$.\\
% the sequence of $r(t)$ and $w(t)$, $t=1,\cdots,T$,
\textbf{Initialization:}\\
\quad\textbf{1.} Set $k=0$. Randomly initialize a positive sequence\\
\qquad\,$(\widetilde{r}^{(0)}(t),\widetilde{w}(t))_{1\leq t\leq T}$.\\
\quad\textbf{2.} Solve optimization problem \eqref{leader2} and obtain the unique global optimal\\
\qquad\,solution $\bm{\nu}_{\mathcal{N}}^*$.\\
\quad\textbf{3.} Calculate Nash equilibrium $\bm{\nu}^{(0)}=\textrm{NE}(\widetilde{r}^{(0)}(t),\widetilde{w}(t))$ (Alg.~\hyperlink{algorithm1}{1}) for\\
\qquad\,the game model \eqref{followers}, and denote the total electricity demand of flexible\\
\qquad\,users in each time slot as\\
\qquad\,$\bm{\nu}_{\mathcal{N}}^{(0)}=(\nu_{\mathcal{N}}^{(0)}(1),\nu_{\mathcal{N}}^{(0)}(2),\cdots,\nu_{\mathcal{N}}^{(0)}(T))^{\top}$.\\
\quad\textbf{4.} Calculate $er=\|\bm{\nu}_{\mathcal{N}}^{(0)}-\bm{\nu}_{\mathcal{N}}^*\|_{\infty}$.\\
\textbf{While not converged do}\\
\quad\textbf{5.} \textbf{for} $t=1,2,\cdots,T$ \textbf{do}\\
\quad\textbf{6.} \quad \textbf{if} $\nu_{\mathcal{N}}^{(k)}(t)>\nu_{\mathcal{N}}^*(t)$, \textbf{then let}\\
\quad\textbf{7.} \quad\quad $\widetilde{r}^{(k+1)}(t)=\widetilde{r}^{(k)}(t)+\varepsilon_0$;\\
\quad\textbf{8.} \quad \textbf{if} $\nu_{\mathcal{N}}^{(k)}(t)<\nu_{\mathcal{N}}^*(t)$, \textbf{then let}\\
\quad\textbf{9.} \quad\quad $\widetilde{r}^{(k+1)}(t)=\widetilde{r}^{(k)}(t)-\varepsilon_0$.\\
\quad\textbf{10.} \textbf{end for}\\
\quad\textbf{11.} Calculate Nash equilibrium $\bm{\nu}^{(k+1)}=\textrm{NE}(\widetilde{r}^{(k+1)}(t),\widetilde{w}(t))$\\
\qquad~~(Alg.~\hyperlink{algorithm1}{1}) and total electricity demand $\bm{\nu}_{\mathcal{N}}^{(k+1)}$ for game model \eqref{followers}.\\
\quad\textbf{12.} Calculate $er=\|\bm{\nu}_{\mathcal{N}}^{(k+1)}-\bm{\nu}_{\mathcal{N}}^*\|_{\infty}$.\\
\quad\textbf{13.} Update $k=k+1$.\\
\textbf{Until} $er<\varepsilon_0$\\
\textbf{Output:}
The optimal pricing sequence $\{\widetilde{r}^{(k)}(t)\}_{1\leq t\leq T}$.\\
\bottomrule
\end{tabular}
\end{center}
\end{table}
}

Next we consider the game model  \eqref{followers} of followers. By \eqref{priceform0} and \eqref{wandr} the price formula can be written as
\begin{equation}\label{priceform01}
 \pi(\nu_{\mathcal{N}}(t),\widetilde{r}(t),\widetilde{w}(t))=\frac{\nu_{\mathcal{N}}(t)+\widetilde{r}(t)}{\widetilde{w}(t)},t=1,\cdots,T.
 \end{equation}
 By Proposition \ref{oristruni}, the model \eqref{followers} with the leader's strategy \eqref{priceform01} admits a strict and unique Nash equilibrium.
Since the analytic expression of the Nash equilibrium is unknown, we try to find its numerical solution.  Algorithm \hyperlink{algorithm1}{1} provides a numerical solution
of model  \eqref{followers}: when we input a sequence $(\widetilde{r}(t),\widetilde{w}(t))_{1\leq t\leq T}$ for \eqref{priceform01}, a  Nash equilibrium of the strategy combination of all flexible users
can be approximatively obtained through Algorithm \hyperlink{algorithm1}{1}.
% We calculate the corresponding Nash equilibrium $\bm{\nu}=\textrm{NE}(\widetilde{r}(t),\widetilde{w}(t))$ for game problem \eqref{followers}. Here we employ the best response algorithm, as shown in Algorithm \hyperlink{algorithm2}{2}.

Finally, we need to find an optimal pricing formula $\pi^*(t)$ to obtain the numerical solution of SE for Stackelberg game \eqref{leader}-\eqref{followers}. That is to say, we manage to find the optimal sequence $(\widetilde{r}^*(t),\widetilde{w}^*(t))_{1\leq t\leq T}$ so that the Nash equilibrium $\bm{\nu}^*$ of game model \eqref{followers} with the leader's strategy \eqref{priceform01} satisfies
\begin{equation}\label{SEnum}
     \sum_{i\in\mathcal{N}}\nu_i^*(t)=\nu_{\mathcal{N}}^*(t), ~~\forall 1\leq t\leq T.
\end{equation}
The detailed search algorithm is shown in Algorithm \hyperlink{algorithm2}{2}. In this algorithm, inspired by \eqref{saddpoi}, we can randomly initialize a positive sequence $(\widetilde{r}(t),\widetilde{w}(t))_{1\leq t\leq T}$, and then regulate only the sequence $\{\widetilde{r}(t)\}_{t=1,\ldots,T}$ to adjust the Nash equilibrium of game model \eqref{followers}, until \eqref{SEnum} approximatively holds.
% until the termination condition is met.
% Last but not least, in each iteration, for sequence $(\widetilde{r}(t),\widetilde{w}(t))$, $t=1,\cdots,T$, we need to calculate the corresponding Nash equilibrium $\bm{\nu}=\textrm{NE}(\widetilde{r}(t),\widetilde{w}(t))$ for game model \eqref{followers}. Here we employ the best response algorithm, as shown in Algorithm \hyperlink{algorithm1}{1}.

%低中高需求的灵活用户各占1/3, 且令\nu_i^{\max}=g_i/(T-1);
%供应需求曲线图(r(t),w(t))、最优定价曲线图(Rtilde(t),Wtilde(t))、leader方最优解曲线与follower方相应均衡下的需求曲线.
We provide a simulation example as follows. Set the number of flexible electricity users as $n=20$,
the number of scheduling time window $T=24$, with each time slot being $1$ hour. Moreover, the total electricity demands of flexible users are $1.55$, $1.49$, $1.04$, $4.13$, $3.7$, $2.03$, $2.29$, $0.92$, $4.47$, $0.81$, $8.98$, $0.02$, $0.29$, $0.37$, $0.13$, $4.64$, $4.65$, $4.83$, $4.49$, and $0.26$, respectively. For any $1\leq i\leq n$, we set $\nu_i^{\max}=2g_i/T$. The regular user electricity consumption and new energy generation sequence $(r(t),w(t))_{1\leq t\leq T}$ are shown in Fig.~\ref{Fig4} (a). It can be verified that the inequalities \eqref{asum1} do not hold in this situation.
% Moreover, assume the total electricity demand $g_i$ of each flexible user obeys an independent and uniform distribution in $(0,12)$, and we set $\nu_i^{\max}=2g_i/T$.
Through Algorithms \hyperlink{algorithm1}{1}-\hyperlink{algorithm2}{2}, we can achieve the numerical solution of SE for model \eqref{leader}-\eqref{followers}, in which the controllable energy generation $\{c(t)\}_{t=1,\ldots,T}$ and the Nash equilibrium $\{\nu_i^*(t)\}_{i=1,\ldots,n;t=1,\ldots,T}$ of game model \eqref{followers} are shown in Fig.~\ref{Fig4} (a). Correspondingly, Fig.~\ref{Fig4} (b) shows the variation curve of the utility company's cost $u^l(\bm{\pi}^{(k)},\bm{\nu}^{(k)})$ with iteration number $k$, from which we can see that the cost function quickly converges to the optimal cost of the utility company, which can be obtained by solving problem \eqref{leader2}.
% which demonstrates the fast convergence of Algorithm \hyperlink{algorithm2}{2}.
% \begin{figure}[ht]
% 	\centering
% 	\subfigure[]{
% 	  \includegraphics[width=1.6in]{Figure/SupplyDemandCurve}
% 	}
% 	\subfigure[]{
% 	  \includegraphics[width=1.6in]{Figure/OptimalPriceCurve}
% 	}
% 	\subfigure[]{
% 	  \includegraphics[width=1.6in]{Figure/ConvergenceCurve}
% 	}
% 	\subfigure[]{
% 	  \includegraphics[width=1.6in]{Figure/StackelbergEquilibrium}
% 	}
% 	\caption{The curve of the electricity demand $r(t)$ and uncontrollable power supply $w(t)$ (a), the optimal sequence $(\widetilde{r}^*(t),\widetilde{w}^*(t))$ and optimal electricity prices $\pi(\nu_{\mathcal{N}}(t),\widetilde{r}^*(t),\widetilde{w}^*(t))$ (b), the variation curve of error value $er$ with iteration number $k$ (c), the global solution $\bm{\nu}_{\mathcal{N}}^*$ and electricity demand patterns of flexible users (d).}
% 	\label{Fig4}
% \end{figure}
\begin{figure}[htb]
     \centering
     \subfigure[]{
	  \includegraphics[width=3.2in]{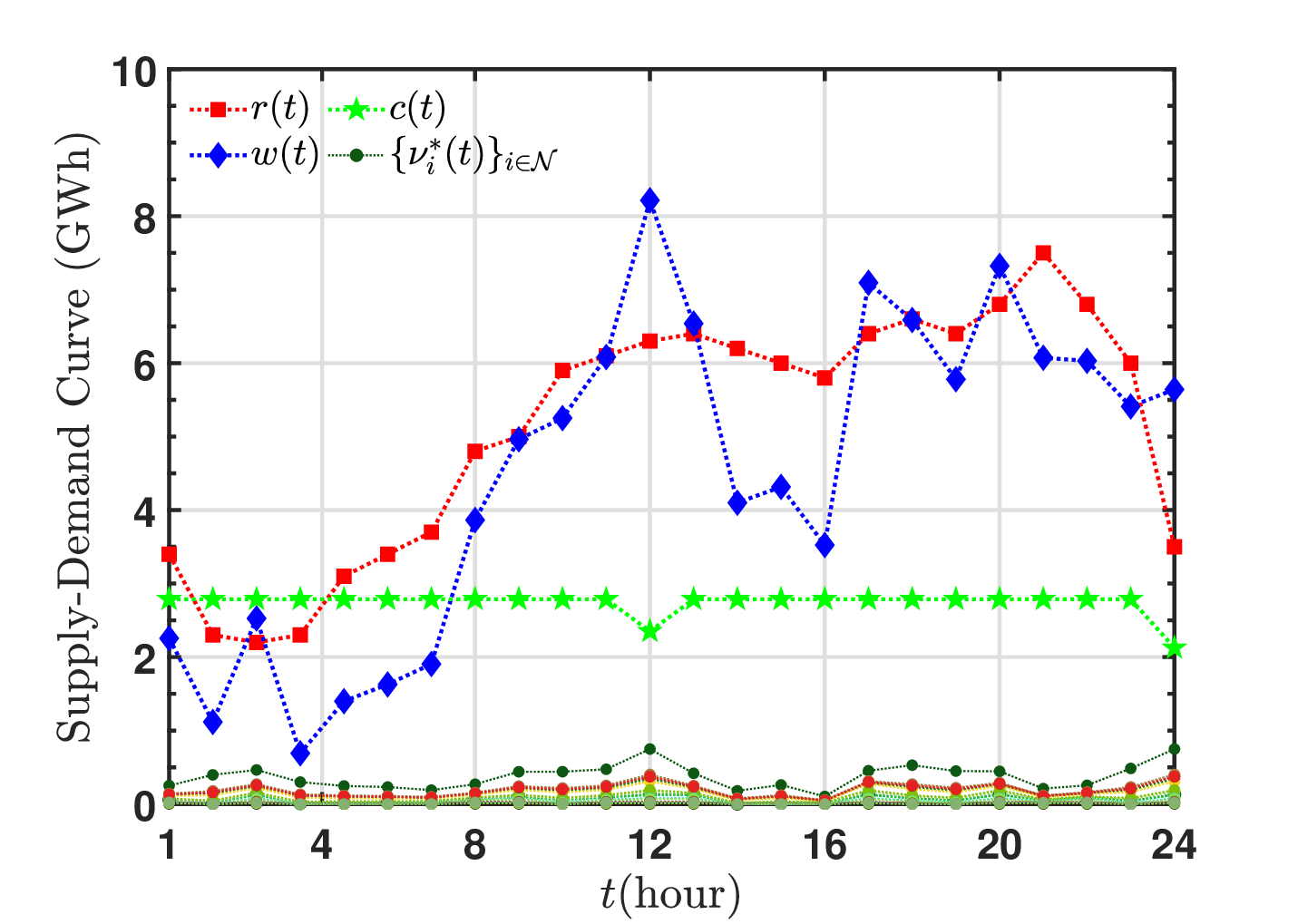}
	}
	\subfigure[]{
	  \includegraphics[width=3.2in]{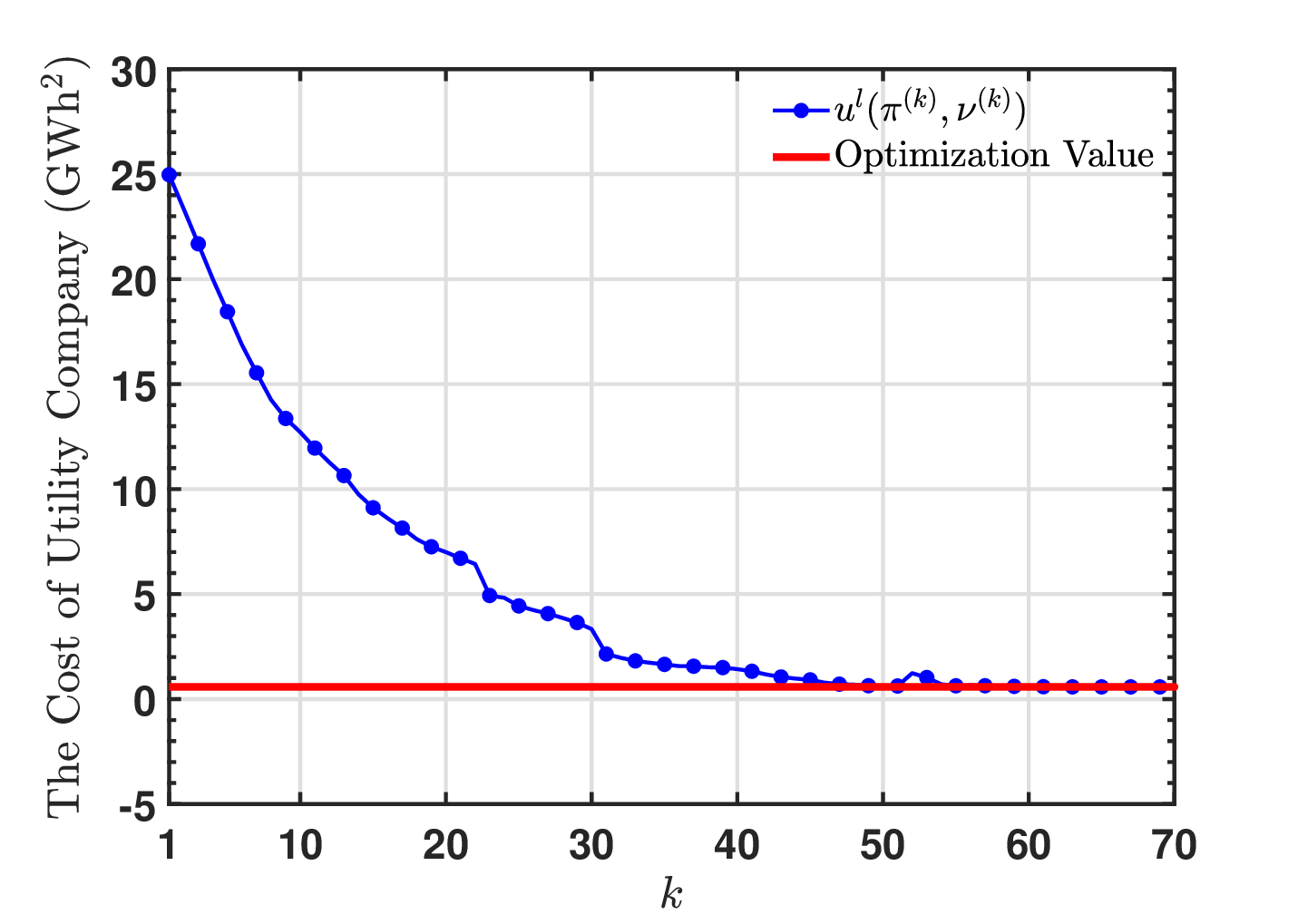}
	}
     \caption{An example of achieving numerical solution for SE through Algorithm 1-2, in which the uncontrollable energy accounts for 62.22\% of energy supply, and flexible users account for 29.41\% of the total energy demand. The curves of $\{w(t),r(t),c(t),\nu_i^*(t)\}$ with respect to time slot $t$ are shown in (a), while the cost function $u^l(\bm{\pi}^{(k)},\bm{\nu}^{(k)})$ with respect to iteration number $k$ is shown in (b).}
     \label{Fig4}
\end{figure}

From above example, for the case when  the condition \eqref{asum1} is not satisfied, the optimal pricing formula for the  utility company  may be achieved by the Algorithms \hyperlink{algorithm1}{1}-\hyperlink{algorithm2}{2}.

% Given the above information, by executing Algorithm \hyperlink{algorithm2}{2}, we can obtain the optimal sequence $(\widetilde{r}^*(t),\widetilde{w}^*(t))_{1\leq t\leq T}$ and the corresponding electricity price sequence $\{\pi(\nu_{\mathcal{N}}(t),\widetilde{r}^*(t),\widetilde{w}^*(t))\}_{t=1,\cdots,T}$, as shown in Fig.~\ref{Fig4} (b). In Algorithm \hyperlink{algorithm2}{2}, we set the tolerance $\varepsilon_0=10^{-6}$. Fig.~\ref{Fig4} (c) shows the variation curve of error value $er$ with iteration number $k$, which demonstrates the fast convergence of Algorithm \hyperlink{algorithm2}{2}.
% Lastly, Fig.~\ref{Fig4} (d) shows the global solution $\bm{\nu}_{\mathcal{N}}^*$ of optimization problem \eqref{leader2}, and the Nash equilibrium $\bm{\nu}^*=\{\nu_i^*(t)\}_{i=1,\ldots,n;t=1,\ldots,T}$ of game model \eqref{followers} under the pricing scheme $\{\pi(\nu_{\mathcal{N}}(t),\widetilde{r}^*(t),\widetilde{w}^*(t))\}_{t=1,\cdots,T}$, which is the electricity demand patterns of flexible users. From Fig.~\ref{Fig4} (d), it can be seen that under the regulation of the pricing scheme, the electricity demands of flexible users remain consistent with the optimal solution $\bm{\nu}_{\mathcal{N}}^*$, and changes in the opposite direction with the fluctuation of electricity prices.

% \section{Simulation Results\label{sec4}}

\renewcommand{\thesection}{\Roman{section}}
\section{Conclusion\label{sec6}}
\renewcommand{\thesection}{\arabic{section}}
% 随着智能电网的飞速发展，RTP 已经成为最重要的需求响应策略之一。在这篇文章中，我们提出了一种Stackelberg博弈方法来建模效用公司与用户之间的互动关系，并在不同的供需关系下分别给出了最优的RTP方案的解析表达式与数值解。利用该定价方案，在用户需求达到纳什均衡的同时，可以充分利用新能源，降低传统发电方式的成本，从而实现社会效益的最大化。

% 这项工作可以从不同方面扩展。例如，我们可以考虑将储能商加入博弈，并且可以考虑当可再生能源发电量为随机变量时的均衡求解问题。这些问题是有意义的并提供了未来有价值的研究方向。
With the rapid development of smart grid, RTP has become one of the most important DR strategies. In this paper, we propose a Stackelberg game model to describe the interaction between utility company and users, and provide analytical expressions and numerical solutions for the optimal pricing formula under different supply and demand relationships. 
Under our optimal pricing scheme, flexible electricity users have an unique and strict Nash equilibrium for electricity demand, in which new energy can be sufficiently utilized and the cost of traditional power generation can be minimized, thereby the social benefit can be maximized.

This work can be extended from different aspects. For example, we can incorporate energy storage companies into the game, consider the equilibrium solving problem when new energy generation is a random variable,
and consider the game behavior between different power suppliers. These questions are meaningful and provide valuable directions for future research.

\appendices
\section{}
% \section{Proof of Lemmas \ref{bianxing}-\ref{global}}\label{appA}
\subsection{Proof of Lemma \ref{bianxing}}\label{appA}
With (\ref{wandr}), the price formula \eqref{priceform0} can be rewritten as
\begin{equation*}
     \pi(t)=\frac{\nu_{\mathcal{N}}(t)+\widetilde{r}(t)}{\widetilde{w}(t)},~~\forall 1\leq t\leq T.
\end{equation*}
Substitute this into  \eqref{followers}, the cost function of each user $i$ can be rewritten as
\begin{equation}
     \begin{aligned}
          u_i^f&(\bm{\pi},\bm{\nu}_i,\bm{\nu}_{-i})=\sum_{t=1}^T\frac{\nu_{\mathcal{N}}(t)+\widetilde{r}(t)}{\widetilde{w}(t)}\nu_i(t)\\
          &=\sum_{t=1}^T\frac{1}{\widetilde{w}(t)}\Big[\nu_i(t)\sum_{j\in\mathcal{N}}\nu_j(t) +\widetilde{r}(t)\nu_i(t)\Big]\\
          &=\sum_{t=1}^T\frac{1}{\widetilde{w}(t)}\bigg\{\nu_i^2(t)+\bigg[\sum_{j\neq i}\nu_j(t)+\widetilde{r}(t)\bigg]\nu_i(t)\bigg\}.
     \end{aligned}
     \label{eq3}
\end{equation}
According to the constraint condition $\bm{1}^{\top}\bm{\nu}_i=g_i$, we can obtain $\nu_i(T)=g_i-\sum_{t=1}^{T-1}\nu_i(t)$. From this and  (\ref{symb1}) we have
\begin{equation*}
     \begin{aligned}
       & \frac{1}{\widetilde{w}(T)}\bigg\{\nu_i^2(T)+\bigg[\sum_{j\neq i}\nu_j(T)+\widetilde{r}(T)\bigg]\nu_i(T)\bigg\}\\
       &~~= \frac{\nu_i^2(T)}{\widetilde{w}(T)}+h_i(T)\nu_i(T) \\
       &~~ =\frac{1}{\widetilde{w}(T)}\bigg[g_i-\sum_{t=1}^{T-1}\nu_i(t)\bigg]^2+h_i(T)\bigg[g_i-\sum_{t=1}^{T-1}\nu_i(t)\bigg].
     \end{aligned}
\end{equation*}
Substitute this into \eqref{eq3} we get
% $\sum_{t=1}^T\nu_i(t)=g_i$
\begin{equation*}
     % \label{mainproof1}
     \begin{aligned}
          &u_i^f(\bm{\pi},\bm{\nu}_i,\bm{\nu}_{-i})\\
          &=\sum_{t=1}^{T-1}\left[\frac{1}{\widetilde{w}(t)}\nu_i^2(t)+h_i(t)\nu_i(t)\right]+\frac{1}{\widetilde{w}(T)}\\
          &~~\times\bigg[g_i-\sum_{t=1}^{T-1}\nu_i(t)\bigg]^2+h_i(T)\bigg[g_i-\sum_{t=1}^{T-1}\nu_i(t)\bigg]\\
                    &=\sum_{t=1}^{T-1}\frac{1}{\widetilde{w}(t)}\nu_i^2(t)+\frac{1}{\widetilde{w}(T)}\bigg[\sum_{t=1}^{T-1}\nu_i(t)\bigg]^2\\
          &~~+\sum_{t=1}^{T-1}\bigg[h_i(t)-h_i(T)-\frac{2g_i}{\widetilde{w}(T)}\bigg]\nu_i(t)+h_i(T)g_i+\frac{g_i^2}{\widetilde{w}(T)}\\
          &=\bm{d}_i^\top \mathbf{C}\bm{d}_i+\bm{\mu}_i^\top\bm{d}_i+h_i(T)g_i+\frac{g_i^2}{\widetilde{w}(T)},
     \end{aligned}
\end{equation*}
where the last line uses \eqref{symb2}.

Now we show the convex property of $u_i^f(\bm{\pi},\bm{\nu}_i,\bm{\nu}_{-i})$.   By \eqref{symb1}-\eqref{symb2}, the vector $\bm{\mu}_i$ does not depend on $\bm{d}_i$, which means
     \begin{equation}\label{grad}
     \frac{\partial \bm{\mu}_i^\top\bm{d}_i}{\partial\bm{d}_i}=\bm{\mu}_i, \frac{\partial^2 \bm{\mu}_i^\top\bm{d}_i}{\partial\bm{d}_i^2}=\mathbf{0},~~~~\forall i\in\mathcal{N}.
     \end{equation}
     Based on \eqref{quafun} and \eqref{grad}, we can get
     \begin{equation}\label{eq7}
          \left\{
               \begin{aligned}
                    &\frac{\partial u_i^f(\bm{\pi},\bm{\nu}_i,\bm{\nu}_{-i})}{\partial\bm{d}_i}=2\mathbf{C}\bm{d}_i+\bm{\mu}_i,\\
                    &\frac{\partial^2 u_i^f(\bm{\pi},\bm{\nu}_i,\bm{\nu}_{-i})}{\partial\bm{d}_i^2}=2\mathbf{C},
               \end{aligned}~~~~\forall i\in\mathcal{N}.
          \right.
     \end{equation}
    Also, by \eqref{symb2} and the condition that $\widetilde{w}(t)>0$ for any $t=1,\ldots,T$,  we have
    \begin{equation*}\label{Csympo}
         \bm{x}^{\top}\mathbf{C}\bm{x}=\sum_{t=1}^{T-1}\frac{1}{\widetilde{w}(t)}x_t^2+\frac{1}{\widetilde{w}(T)}\bigg[\sum_{t=1}^{T-1}x_t\bigg]^2>0
    \end{equation*}
for any $\bm{x}=(x_1,x_2,\ldots,x_{T-1})^{\top}\in\mathbb{R}^{T-1}$ with $\bm{x}\ne\bm{0}$, which means $\mathbf{C}$ is symmetric positive definite. From this and \eqref{eq7},
the objective function $u_i^f(\bm{\pi},\bm{\nu}_i,\bm{\nu}_{-i})$ is a strictly convex quadratic function.

\subsection{Proof of Lemma \ref{les}}\label{appB}

For each follower $i$, by \eqref{symb1} and $\bm{1}^{\top}\bm{\nu}_j=g_j, j=1,\ldots,n$, we have
\begin{equation}\label{hiT}
     \begin{aligned}
          h_i(T)&=\frac{1}{\widetilde{w}(T)}\big[\sum_{j\neq i}\nu_j(T)+\widetilde{r}(T)\big]\\
                &=\frac{1}{\widetilde{w}(T)}\sum_{j\neq i}\big[g_j-\sum_{t=1}^{T-1}\nu_j(t)\big]+\frac{\widetilde{r}(T)}{\widetilde{w}(T)}\\
                &=\frac{g_{\mathcal{N}}-g_i+\widetilde{r}(T)}{\widetilde{w}(T)}-\frac{1}{\widetilde{w}(T)}\big[\sum_{j\neq i}\sum_{t=1}^{T-1}\nu_j(t)\big].
     \end{aligned}
\end{equation}
Combining \eqref{symb2} and \eqref{hiT}, $\bm{\mu}_i$ can be written as
\begin{equation}\label{muchange}
     \begin{aligned}
          \bm{\mu}_i&=
          \begin{pmatrix}
               h_i(1)\\
               \vdots\\
               h_i(T-1)\\
          \end{pmatrix}-\big[h_i(T)+\frac{2g_i}{\widetilde{w}(T)}\big]\bm{1}\\
                    &=
                    \begin{pmatrix}
                         \frac{1}{\widetilde{w}(1)} & \cdots & 0\\
                              \vdots & \ddots & \vdots\\
                              0 & \cdots & \frac{1}{\widetilde{w}(T-1)}\\
                    \end{pmatrix}
                    \begin{pmatrix}
                         \sum_{j\neq i}\nu_j(1)\\
                         \vdots\\
                         \sum_{j\neq i}\nu_j(T-1)\\
                    \end{pmatrix}\\
                    &\quad+
                    \begin{pmatrix}
                         \frac{\widetilde{r}(1)}{\widetilde{w}(1)}\\
                         \vdots\\
                         \frac{\widetilde{r}(T-1)}{\widetilde{w}(T-1)}\\
                    \end{pmatrix}+\frac{1}{\widetilde{w}(T)}\big[\sum_{j\neq i}\sum_{t=1}^{T-1}\nu_j(t)\big]\bm{1}\\
                    &\quad-\frac{g_{\mathcal{N}}+g_i+\widetilde{r}(T)}{\widetilde{w}(T)}\bm{1}\\
                    &=\left\{
                    \begin{pmatrix}
                         \frac{1}{\widetilde{w}(1)} & \cdots & 0\\
                              \vdots & \ddots & \vdots\\
                              0 & \cdots & \frac{1}{\widetilde{w}(T-1)}\\
                    \end{pmatrix}+\frac{1}{\widetilde{w}(T)}\mathbf{1}\mathbf{1}^{\top}\right\}\\
                    &\quad\times
                    \begin{pmatrix}
                         \sum_{j\neq i}\nu_j(1)\\
                         \vdots\\
                         \sum_{j\neq i}\nu_j(T-1)\\
                    \end{pmatrix}+
                    \begin{pmatrix}
                         \frac{\widetilde{r}(1)}{\widetilde{w}(1)}\\
                         \vdots\\
                         \frac{\widetilde{r}(T-1)}{\widetilde{w}(T-1)}\\
                    \end{pmatrix}\\
                    &\quad-\frac{g_{\mathcal{N}}+g_i+\widetilde{r}(T)}{\widetilde{w}(T)}\bm{1}\\
                    &=\mathbf{C}\sum_{j\neq i}\bm{d}_j-\widetilde{\bm{g}}_i,
     \end{aligned}
\end{equation}
where $\widetilde{\bm{g}}_i=\frac{g_{\mathcal{N}}+g_i+\widetilde{r}(T)}{\widetilde{w}(T)}\bm{1}-(\frac{\widetilde{r}(1)}{\widetilde{w}(1)},\ldots,\frac{\widetilde{r}(T-1)}{\widetilde{w}(T-1)})^{\top}$, and
the last equality uses the fact
\begin{multline*}
\big(\sum_{j\neq i}\nu_j(1),\ldots,\sum_{j\neq i}\nu_j(T-1)\big)^{\top}\\
=\sum_{j\neq i}\big(\nu_j(1),\ldots,\nu_j(T-1)\big)^{\top}=\sum_{j\neq i}\bm{d}_j.
\end{multline*}
Using \eqref{muchange}, Eq.~\eqref{eq7_1} can be rewritten as
\begin{equation}\label{eq7_3a}
     % \label{eq7_3}
     \mathbf{C}\Big[2\bm{\bar{d}}_i+\sum_{j\neq i}\bm{d}_j\Big]=\widetilde{\bm{g}}_i.
\end{equation}
% Uniting all followers through \eqref{eq7_3}, we can obtain the linear equation system \eqref{eq7_2}.
% \begin{equation*}
%      \begin{pmatrix}
%           2\mathbf{C} & \mathbf{C} & \cdots & \mathbf{C}\\
%           \mathbf{C} & 2\mathbf{C} & \cdots & \mathbf{C}\\
%           \vdots & \vdots & \ddots & \vdots\\
%           \mathbf{C} & \mathbf{C} & \cdots & 2\mathbf{C}\\
%      \end{pmatrix}
%      \begin{pmatrix}
%           \bm{\bar{d}}_1\\
%           \bm{\bar{d}}_2\\
%           \vdots\\
%           \bm{\bar{d}}_n\\
%      \end{pmatrix}=
%      \begin{pmatrix}
%           \widetilde{\bm{g}}_1\\
%           \widetilde{\bm{g}}_2\\
%           \vdots\\
%           \widetilde{\bm{g}}_n\\
%      \end{pmatrix}.
% \end{equation*}
% Based on \eqref{eq7_1}, further we can get
% \begin{equation}
%      \mathbf{C}\Big[2\bm{\bar{d}}_i+\sum_{j\neq i}\bm{d}_j\Big]=\widetilde{\bm{g}}_i,
% \end{equation}

By uniting all followers through \eqref{eq7_3a}, we can obtain the following linear equation system
% \begin{equation}\label{eq7_2}
%      \begin{pmatrix}
%           2\mathbf{C} & \mathbf{C} & \cdots & \mathbf{C}\\
%           \mathbf{C} & 2\mathbf{C} & \cdots & \mathbf{C}\\
%           \vdots & \vdots & \ddots & \vdots\\
%           \mathbf{C} & \mathbf{C} & \cdots & 2\mathbf{C}\\
%      \end{pmatrix}
%      \begin{pmatrix}
%           \bm{\bar{d}}_1\\
%           \bm{\bar{d}}_2\\
%           \vdots\\
%           \bm{\bar{d}}_n\\
%      \end{pmatrix}=
%      \begin{pmatrix}
%           \widetilde{\bm{g}}_1\\
%           \widetilde{\bm{g}}_2\\
%           \vdots\\
%           \widetilde{\bm{g}}_n\\
%      \end{pmatrix},
% \end{equation}
% where $\widetilde{\bm{g}}_i=\frac{g_{\mathcal{N}}+g_i+\widetilde{r}(T)}{\widetilde{w}(T)}\bm{1}-(\frac{\widetilde{r}(1)}{\widetilde{w}(1)},\ldots,\frac{\widetilde{r}(T-1)}{\widetilde{w}(T-1)})^{\top}$ (The detailed derivation process of \eqref{eq7_2} is postponed to Appendix \ref{appB}).
% Expand (20) yields
\begin{eqnarray}\label{eq7_2a}
\begin{aligned}
\mathbf{C}\bm{\bar{d}}_1&+\mathbf{C}\left(\bm{\bar{d}}_1+\cdots+\bm{\bar{d}}_n\right)= \widetilde{\bm{g}}_1\\
&~\vdots\\
     \mathbf{C}\bm{\bar{d}}_n&+\mathbf{C}\left(\bm{\bar{d}}_1+\cdots+\bm{\bar{d}}_n\right)= \widetilde{\bm{g}}_n
\end{aligned}.
\end{eqnarray}
Sum up all rows of \eqref{eq7_2a} yields
$$\mathbf{C}\left(\bm{\bar{d}}_1+\cdots+\bm{\bar{d}}_n\right)=\frac{1}{n+1}\left(\widetilde{\bm{g}}_1+\cdots+\widetilde{\bm{g}}_n\right).$$
Substituting this into \eqref{eq7_2a} we have
\begin{equation}\label{eq7_2b}
     \bm{\bar{d}}_i=\mathbf{C}^{-1}\left[\widetilde{\bm{g}}_i-\frac{1}{n+1}\left(\widetilde{\bm{g}}_1+\cdots+\widetilde{\bm{g}}_n\right)\right],~~\forall i\in\mathcal{N}.
\end{equation}
According to the Sherman-Morrison formula\footnotemark, we have
\begin{eqnarray}\label{Cinva}
&&\mathbf{C}^{-1}=\mbox{diag}(\widetilde{w}(1),\ldots,\widetilde{w}(T-1))-\frac{1}{1+\frac{\sum_{t=1}^{T-1}\widetilde{w}(t)}{\widetilde{w}(T)}}\nonumber\\
&&~~\times\frac{1}{\widetilde{w}(T)}(\widetilde{w}(1),\ldots,\widetilde{w}(T-1))^{\top}(\widetilde{w}(1),\ldots,\widetilde{w}(T-1))\nonumber\\
&&=\mbox{diag}(\widetilde{w}(1),\ldots,\widetilde{w}(T-1))-\frac{1}{\sum_{t=1}^{T}\widetilde{w}(t)}\\
&&~~\times(\widetilde{w}(1),\ldots,\widetilde{w}(T-1))^{\top}(\widetilde{w}(1),\ldots,\widetilde{w}(T-1)).\nonumber
\end{eqnarray}
\footnotetext{Sherman-Morrison formula \cite{hager1989updating}: Given $\mathbf{A}\in\mathbb{R}^{n\times n}$, $\bm{u}, \bm{v}\in\mathbb{R}^n$, such that $\det(\mathbf{A})\ne 0$, $1+\bm{v}^{\top}\mathbf{A}^{-1}\bm{u}\ne 0$, then $\mathbf{A}+\bm{u}\bm{v}^{\top}$ is invertible, and we have $$(\mathbf{A}+\bm{u}\bm{v}^{\top})^{-1}=\mathbf{A}^{-1}-\frac{\mathbf{A}^{-1}\bm{u}\bm{v}^{\top}\mathbf{A}^{-1}}{1+\bm{v}^{\top}\mathbf{A}^{-1}\bm{u}}.$$}
On the other hand, we have
\begin{equation}\label{bmgi}
     \begin{aligned}
          &~~~\widetilde{\bm{g}}_i-\frac{1}{n+1}\sum_{i=1}^n\widetilde{\bm{g}}_i\\
          &=\frac{g_{\mathcal{N}}+g_i+\widetilde{r}(T)}{\widetilde{w}(T)}\bm{1}-
          \begin{pmatrix}
               \frac{\widetilde{r}(1)}{\widetilde{w}(1)}\\
               \vdots\\
               \frac{\widetilde{r}(T-1)}{\widetilde{w}(T-1)}\\
          \end{pmatrix}\\
          &~~~-\frac{(n+1)g_{\mathcal{N}}+n\widetilde{r}(T)}{(n+1)\widetilde{w}(T)}\bm{1}+\frac{n}{n+1}
          \begin{pmatrix}
               \frac{\widetilde{r}(1)}{\widetilde{w}(1)}\\
               \vdots\\
               \frac{\widetilde{r}(T-1)}{\widetilde{w}(T-1)}\\
          \end{pmatrix}\\
          &=\frac{(n+1)g_i+\widetilde{r}(T)}{(n+1)\widetilde{w}(T)}\bm{1}-\frac{1}{n+1}
          \begin{pmatrix}
               \frac{\widetilde{r}(1)}{\widetilde{w}(1)}\\
               \vdots\\
               \frac{\widetilde{r}(T-1)}{\widetilde{w}(T-1)}\\
          \end{pmatrix},~~\forall i\in\mathcal{N}.
     \end{aligned}
\end{equation}
Substituting \eqref{Cinva} and \eqref{bmgi} into \eqref{eq7_2b}, we can get
\begin{equation}\label{dibar}
     \begin{aligned}
          \bm{\bar{d}}_i&=\frac{(n+1)g_i+\widetilde{r}(T)}{(n+1)\widetilde{w}(T)}
          \begin{pmatrix}
               \widetilde{w}(1)\\
               \vdots\\
               \widetilde{w}(T-1)\\
          \end{pmatrix}-\frac{1}{n+1}\\
          &\times
          \begin{pmatrix}
               \widetilde{r}(1)\\
               \vdots\\
               \widetilde{r}(T-1)\\
          \end{pmatrix}-\frac{\sum_{t=1}^{T-1}\widetilde{w}(t)}{\sum_{t=1}^{T}\widetilde{w}(t)}\frac{(n+1)g_i+\widetilde{r}(T)}{(n+1)\widetilde{w}(T)}\\
          &\times
          \begin{pmatrix}
          \widetilde{w}(1)\\
          \vdots\\
          \widetilde{w}(T-1)\\
          \end{pmatrix}
          +\frac{\sum_{t=1}^{T-1}\widetilde{r}(t)}{\sum_{t=1}^{T}\widetilde{w}(t)}\frac{1}{n+1}
          \begin{pmatrix}
               \widetilde{w}(1)\\
               \vdots\\
               \widetilde{w}(T-1)\\
          \end{pmatrix}\\
          &=\left[\frac{\widetilde{w}(T)}{\sum_{t=1}^{T}\widetilde{w}(t)}\frac{(n+1)g_i+\widetilde{r}(T)}{(n+1)\widetilde{w}(T)}+\frac{\sum_{t=1}^{T-1}\widetilde{r}(t)}{(n+1)\sum_{t=1}^{T}\widetilde{w}(t)}\right]\\
          &\times
          \begin{pmatrix}
               \widetilde{w}(1)\\
               \vdots\\
               \widetilde{w}(T-1)\\
          \end{pmatrix}-\frac{1}{n+1}
          \begin{pmatrix}
               \widetilde{r}(1)\\
               \vdots\\
               \widetilde{r}(T-1)\\
          \end{pmatrix}\\
          &=\left[\frac{g_i}{\sum_{t=1}^{T}\widetilde{w}(t)}+\frac{\sum_{t=1}^{T}\widetilde{r}(t)}{(n+1)\sum_{t=1}^{T}\widetilde{w}(t)}\right]
          \begin{pmatrix}
               \widetilde{w}(1)\\
               \vdots\\
               \widetilde{w}(T-1)\\
          \end{pmatrix}\\
          &-\frac{\sum_{t=1}^{T}\widetilde{w}(t)}{(n+1)\sum_{t=1}^{T}\widetilde{w}(t)}
          \begin{pmatrix}
               \widetilde{r}(1)\\
               \vdots\\
               \widetilde{r}(T-1)\\
          \end{pmatrix}
     \end{aligned}
\end{equation}
\begin{equation*}
     \begin{aligned}
          &=\frac{g_i}{\sum_{t=1}^{T}\widetilde{w}(t)}
          \begin{pmatrix}
               \widetilde{w}(1)\\
               \vdots\\
               \widetilde{w}(T-1)\\
          \end{pmatrix}+\frac{1}{(n+1)\sum_{t=1}^{T}\widetilde{w}(t)}\\
          &\times
          \begin{pmatrix}
               \sum_{t=1}^{T}[\widetilde{r}(t)\widetilde{w}(1)-\widetilde{w}(t)\widetilde{r}(1)]\\
               \vdots\\
               \sum_{t=1}^{T}[\widetilde{r}(t)\widetilde{w}(T-1)-\widetilde{w}(t)\widetilde{r}(T-1)]\\
          \end{pmatrix}\\
          &=(\nu_i(1),\nu_i(2),\ldots,\nu_i(T-1))^{\top},~~\forall i\in\mathcal{N}.
     \end{aligned}
\end{equation*}
Besides, we have
\begin{multline*}
          \sum_{t=1}^T\sum_{s=1}^T\widetilde{w}(s)\widetilde{w}(t)\left[\frac{\widetilde{r}(s)}{\widetilde{w}(s)}-\frac{\widetilde{r}(t)}{\widetilde{w}(t)}\right]\\
          =\sum_{t=1}^T\widetilde{w}(t)\sum_{s=1}^T\widetilde{r}(s)-\sum_{t=1}^T\widetilde{r}(t)\sum_{s=1}^T\widetilde{w}(s)=0,
\end{multline*}
so
\begin{multline}\label{nuiT}
     \sum_{t=1}^{T-1}\sum_{s=1}^T\widetilde{w}(s)\widetilde{w}(t)\left[\frac{\widetilde{r}(s)}{\widetilde{w}(s)}-\frac{\widetilde{r}(t)}{\widetilde{w}(t)}\right]\\
     =-\sum_{s=1}^T\widetilde{w}(s)\widetilde{w}(T)\left[\frac{\widetilde{r}(s)}{\widetilde{w}(s)}-\frac{\widetilde{r}(T)}{\widetilde{w}(T)}\right].
\end{multline}
According to \eqref{dibar} and the constraint condition $ \sum_{t=1}^T\nu_i(t)=g_i$, we can obtain
\begin{equation*}
     \begin{aligned}
          \nu_i(T)&=g_i-\sum_{t=1}^{T-1}\nu_i(t)\\
                  &=g_i-\frac{\sum_{t=1}^{T-1}\widetilde{w}(t)}{\sum_s\widetilde{w}(s)}g_i-\frac{1}{(n+1)\sum_s\widetilde{w}(s)}\\
                  &\qquad\qquad\times\sum_{t=1}^{T-1}\sum_{s=1}^T\widetilde{w}(s)\widetilde{w}(t)\left[\frac{\widetilde{r}(s)}{\widetilde{w}(s)}-\frac{\widetilde{r}(t)}{\widetilde{w}(t)}\right]\\
                  &=\frac{\widetilde{w}(T)}{\sum_s\widetilde{w}(s)}g_i+\frac{1}{(n+1)\sum_s\widetilde{w}(s)}\\
                  &\qquad\qquad\times\sum_{s=1}^T\widetilde{w}(s)\widetilde{w}(T)\left[\frac{\widetilde{r}(s)}{\widetilde{w}(s)}-\frac{\widetilde{r}(T)}{\widetilde{w}(T)}\right], ~~\forall i\in\mathcal{N},
     \end{aligned}
\end{equation*}
where the last equality uses \eqref{nuiT}.

\subsection{Proof of Lemma \ref{conmap}}\label{appC}

% For each user $i$, define its strategy hyperplane by $\mathbb{H}_i:=\{\bm{x}\in\mathbb{R}^T: \bm{1}^{\top} \bm{x}=g_i\}$, and we denote $\mathbb{H}_{-i}:=\mathbb{H}_1\times\cdots\times \mathbb{H}_{i-1}\times\mathbb{H}_{i+1}\times\cdots\times \mathbb{H}_{n}$ and $\mathbb{H}:=\mathbb{H}_1\times\cdots\times\mathbb{H}_{n}$. By Lemma \ref{bianxing} and \eqref{eq7_1}, we know that for each user $i$ and any strategy combination $\bm{\nu}_{-i}\in\mathbb{H}_{-i}$ of other users, the best response strategy of user $i$ in $\mathbb{H}_i$ is unique.
% Thus, we can define the unique best response strategy by
% \begin{equation*}
%   \begin{aligned}
%     \bm{\bar{\nu}}_{i}=\bm{\bar{\nu}}_{i}(\bm{\nu}_{-i}):&=\mathop{\arg\min}_{\bm{\nu}_i\in\mathbb{H}_i}u_i^f(\bm{\pi},\bm{\nu}_i,\bm{\nu}_{-i}),\\
%                                                          &\qquad\qquad\qquad\qquad\forall i\in\mathcal{N}, \bm{\nu}_{-i}\in\mathbb{H}_{-i},
%   \end{aligned}
% \end{equation*}
% and we have
Before the proof of Lemma \ref{conmap}, we need to introduce some lemmas as follows.
\begin{lemma}\label{global}
     For each user $i\in\mathcal{N}$ and any strategy combination $\bm{\nu}_{-i}\in\mathbb{H}_{-i}$, the best response strategy $\bm{\bar{\nu}}_{i}=(\bar{\nu}_i(1),\ldots,\bar{\nu}_i(T))^{\top}:=\mathop{\arg\min}_{\bm{\nu}_i\in\mathbb{H}_i}u_i^f(\bm{\pi},\bm{\nu}_i,\bm{\nu}_{-i})$ is
     \begin{multline}\label{expli}
          \bm{\bar{\nu}}_{i}=-\frac{1}{2}\sum_{j\neq i}\bm{\nu}_{j}+\frac{g_{\mathcal{N}}+g_i+\sum_s\widetilde{r}(s)}{2\sum_s\widetilde{w}(s)}\\
          \times\begin{pmatrix}
               \widetilde{w}(1)\\
               \vdots\\
               \widetilde{w}(T)\\
          \end{pmatrix}-\frac{1}{2}
          \begin{pmatrix}
               \widetilde{r}(1)\\
               \vdots\\
               \widetilde{r}(T)\\
          \end{pmatrix}.
     \end{multline}
     % the global minimum point $\bm{\bar{\nu}}_{i}=(\bar{\nu}_i(1),\ldots,\bar{\nu}_i(T))^{\top}\in\mathbb{H}_i$ of the quadratic function \eqref{quafun} is
     % \begin{equation}\label{expli}
     %      \bar{\nu}_i(t)=\frac{g_{\mathcal{N}}+g_i}{2g_{\mathcal{N}}}\widetilde{w}^*(t)-\frac{1}{2}\sum_{j\neq i}\nu_j(t), t=1,\ldots,T,
     % \end{equation}
     % that is, the best response strategy $\bm{\bar{\nu}}_{i}=\bm{\bar{\nu}}_{i}(\bm{\nu}_{-i})$.
\end{lemma}
\begin{proof}
     For each user $i$, by \eqref{eq7_1} the unique global minimum point $\bm{\bar{d}}_{i}$ satisfies
     \begin{equation}\label{bardi}
          \bm{\bar{d}}_i=-\frac{1}{2}\mathbf{C}^{-1}\bm{\mu}_i.
     \end{equation}
     Substituting \eqref{muchange} and \eqref{Cinva} into \eqref{bardi}, we have
     % we provide an explicit expression for the global minimum point $\mathbf{\bar{D}}_i$.
     % Firstly, we can rewrite matrix $\mathbf{C}^*$ as $$\mathbf{C}^*=\mbox{diag}\Big(\frac{1}{\widetilde{w}^*(1)},\ldots,\frac{1}{\widetilde{w}^*(T-1)}\Big)+\frac{1}{\widetilde{w}^*(T)}\mathbf{1}\mathbf{1}^{\top}.$$
     \begin{equation*}
          \begin{aligned}
               &~~~\bm{\bar{d}}_i=-\frac{1}{2}\sum_{j\neq i}\bm{d}_{j}+\frac{1}{2}\mathbf{C}^{-1}\widetilde{\bm{g}}_i\\
               &=-\frac{1}{2}\sum_{j\neq i}\bm{d}_{j}+\frac{g_{\mathcal{N}}+g_i+\widetilde{r}(T)}{2\widetilde{w}(T)}
               \begin{pmatrix}
                    \widetilde{w}(1)\\
                    \vdots\\
                    \widetilde{w}(T-1)\\
               \end{pmatrix}\\
               &~~-\frac{1}{2}
               \begin{pmatrix}
                    \widetilde{r}(1)\\
                    \vdots\\
                    \widetilde{r}(T-1)\\
               \end{pmatrix}
               -\frac{g_{\mathcal{N}}+g_i+\widetilde{r}(T)}{2\widetilde{w}(T)}\frac{\sum_{s=1}^{T-1}\widetilde{w}(s)}{\sum_s\widetilde{w}(s)}\\
               &~~\times
               \begin{pmatrix}
               \widetilde{w}(1)\\
               \vdots\\
               \widetilde{w}(T-1)\\
               \end{pmatrix}+\frac{\sum_{s=1}^{T-1}\widetilde{r}(s)}{2\sum_s\widetilde{w}(s)}
               \begin{pmatrix}
                    \widetilde{w}(1)\\
                    \vdots\\
                    \widetilde{w}(T-1)\\
               \end{pmatrix}\\
               &=-\frac{1}{2}\sum_{j\neq i}\bm{d}_{j}+\frac{g_{\mathcal{N}}+g_i+\sum_s\widetilde{r}(s)}{2\sum_s\widetilde{w}(s)}\\
               &~~\times\begin{pmatrix}
                    \widetilde{w}(1)\\
                    \vdots\\
                    \widetilde{w}(T-1)\\
               \end{pmatrix}-\frac{1}{2}
               \begin{pmatrix}
                    \widetilde{r}(1)\\
                    \vdots\\
                    \widetilde{r}(T-1)\\
               \end{pmatrix}.
          \end{aligned}
     \end{equation*}
     Besides, due to $\bm{1}^{\top}\bm{\bar{\nu}}_{i}=g_i$, we have
     \begin{equation*}
          \begin{aligned}
               &~~~\bar{\nu}_i(T)=g_i-\bm{1}^{\top}\bm{\bar{d}}_i\\
               &=g_i+\frac{1}{2}\sum_{j\neq i}\sum_{t=1}^{T-1}\nu_j(t)-\frac{(g_{\mathcal{N}}+g_i)\sum_{s=1}^{T-1}\widetilde{w}(s)}{2\sum_s\widetilde{w}(s)}\\
               &~~-\left[\frac{\sum_s\widetilde{r}(s)\sum_{s=1}^{T-1}\widetilde{w}(s)}{2\sum_s\widetilde{w}(s)}-\frac{\sum_s\widetilde{w}(s)\sum_{s=1}^{T-1}\widetilde{r}(s)}{2\sum_s\widetilde{w}(s)}\right]\\
               &=g_i+\frac{1}{2}\sum_{j\neq i}[g_j-\nu_j(T)]-\frac{(g_{\mathcal{N}}+g_i)\sum_{s=1}^{T-1}\widetilde{w}(s)}{2\sum_s\widetilde{w}(s)}\\
               &~~+\frac{\sum_s\widetilde{r}(s)\widetilde{w}(T)-\sum_s\widetilde{w}(s)\widetilde{r}(T)}{2\sum_s\widetilde{w}(s)}\\
               &=-\frac{1}{2}\sum_{j\neq i}\nu_{j}(T)+g_i+\frac{1}{2}(g_{\mathcal{N}}-g_i)\\
               &~~-\frac{(g_{\mathcal{N}}+g_i)\sum_{s=1}^{T-1}\widetilde{w}(s)}{2\sum_s\widetilde{w}(s)}+\frac{\sum_s\widetilde{r}(s)\widetilde{w}(T)}{2\sum_s\widetilde{w}(s)}-\frac{1}{2}\widetilde{r}(T)\\
               &=-\frac{1}{2}\sum_{j\neq i}\nu_{j}(T)
               +\frac{g_{\mathcal{N}}+g_i+\sum_s\widetilde{r}(s)}{2\sum_s\widetilde{w}(s)}\widetilde{w}(T)-\frac{1}{2}\widetilde{r}(T),
          \end{aligned}
     \end{equation*}
     where the third equality uses \eqref{nuiT}.
     So the best response strategy is as \eqref{expli}.
\end{proof}

For any matrix  $\mathbf{A}\in\mathbb{C}^{n\times n}$, let $\|\mathbf{A}\|_1:=\max_j\sum_{i=1}^n |\mathbf{A}_{ij}|$ denote the  maximum absolute column sum norm.

% We show that the best response strategy  in $\mathbb{H}_i$  of each user $i$ is unique if the leader's strategy is $\bm{\pi}^*$.
% \begin{proposition}\label{BSuni}
% When the leader's strategy is $\bm{\pi}^*$, if the inequality \eqref{asum1} holds
% then for each user $i$ and any strategy combination $\bm{\nu}_{-i}\in\mathbb{H}_{-i}$ of other users,
% the best response  strategy of user $i$ in $\mathbb{H}_i$ is unique, i.e., there is a unique $\bm{\bar{\nu}}_i\in\mathbb{H}_i$ such that
%   $$u_i^f(\bm{\pi}^*,\bm{\bar{\nu}}_i,\bm{\nu}_{-i})=\min_{\bm{\nu}_i\in\mathbb{H}_i}u_i^f(\bm{\pi}^*,\bm{\nu}_i,\bm{\nu}_{-i}).$$
% \end{proposition}

% By Proposition \ref{BSuni},
% when the leader's strategy is $\bm{\pi}^*$, we can define the unique best response strategy by
% \begin{equation*}
%   \begin{aligned}
%     \bm{\bar{\nu}}_{i}=\bm{\bar{\nu}}_{i}(\bm{\nu}_{-i}):&=\mathop{\arg\min}_{\bm{\nu}_i\in\mathbb{H}_i}u_i^f(\bm{\pi}^*,\bm{\nu}_i,\bm{\nu}_{-i}),\\
%                                                          &\qquad\qquad\qquad\qquad\forall i\in\mathcal{N}, \bm{\nu}_{-i}\in\mathbb{H}_{-i}.
%   \end{aligned}
% \end{equation*}
\begin{lemma}\label{Spectral}
     For any $\mathbf{A}\in\mathbb{C}^{n\times n}$ and any constant $\varepsilon>0$, we can find an invertible matrix $\mathbf{S}_{\mathbf{A},\varepsilon}\in\mathbb{C}^{n\times n}$ such that $$\|\mathbf{S}_{\mathbf{A},\varepsilon}^{-1}\mathbf{A}\mathbf{S}_{\mathbf{A},\varepsilon}\|_1\leq\rho(\mathbf{A})+\varepsilon,$$ where $\rho(\mathbf{A})$ is the spectral radius of $\mathbf{A}$.
\end{lemma}
\begin{proof}
     For any matrix $\mathbf{A}$, there exists an invertible matrix $\mathbf{P}$ such that $\mathbf{P}^{-1}\mathbf{A}\mathbf{P}=\mathbf{J}$, and $\mathbf{J}$ is the Jordan canonical form of $\mathbf{A}$. Let $\lambda_1,\lambda_2,\ldots,\lambda_n$ be the $n$ eigenvalues of $\mathbf{A}$, so $\mathbf{J}$ can be written as $$\mathbf{J}=\mathbf{\Lambda}+\mathbf{\tilde{I}},$$ where $\mathbf{\Lambda}=\mbox{diag}(\lambda_1,\lambda_2,\ldots,\lambda_n)$ and $$\mathbf{\tilde{I}}=
     \begin{pmatrix}
          0 & \delta_1 &  &  & \\
            & 0 & \delta_2 &  & \\
           &  & \ddots & \ddots & \\
           &  &  & 0 & \delta_{n-1} \\
           &  &  &  & 0\\
     \end{pmatrix}
     (\delta_i=0~\mbox{or}~1).$$
     Let $\mathbf{H}:=\mbox{diag}(1,\varepsilon,\ldots,\varepsilon^{n-1})$, and we have $$(\mathbf{PH})^{-1}\mathbf{A}(\mathbf{PH})=\mathbf{H}^{-1}\mathbf{JH}=\mathbf{H}^{-1}\mathbf{\Lambda H}+\mathbf{H}^{-1}\mathbf{\tilde{I}H}=\mathbf{\Lambda}+\varepsilon\mathbf{\tilde{I}}.$$
     Finally, we make $\mathbf{S}_{\mathbf{A},\varepsilon}=\mathbf{PH}$, then $\mathbf{S}_{\mathbf{A},\varepsilon}$ is invertible, and $$\|\mathbf{S}_{\mathbf{A},\varepsilon}^{-1}\mathbf{A}\mathbf{S}_{\mathbf{A},\varepsilon}\|_1=\|\mathbf{\Lambda}+\varepsilon\mathbf{\tilde{I}}\|_1\leq \max_i |\lambda_i|+\varepsilon=\rho(\mathbf{A})+\varepsilon.$$
\end{proof}

\emph{Proof of Lemma \ref{conmap}:} First, we provide an explicit expression for the mapping $\bm{F}$.
     % We first define the ``sequential'' best response mapping $F: \mathbb{H}\rightarrow\mathbb{H}$ on $\mathbb{H}$. Concretely, for any strategy combination $\bm{\nu}=(\bm{\nu}_1,\bm{\nu}_2,\ldots,\bm{\nu}_n)^{\top}\in\mathbb{H}$, we make $$F(\bm{\nu})=(\bm{\bar{\nu}}_{1},\bm{\bar{\nu}}_{2},\ldots,\bm{\bar{\nu}}_{n})^{\top},$$ where
     % \begin{equation}\label{seqbs}
     %      \bm{\bar{\nu}}_{i}=\bm{\bar{\nu}}_{i}(\bm{\bar{\nu}}_1,\ldots,\bm{\bar{\nu}}_{i-1},\bm{\nu}_{i+1},\ldots,\bm{\nu}_{n}), ~~i=1,\ldots,n.
     % \end{equation}
     % We can get that $F(\bm{\nu})$ still belongs to $\mathbb{H}$, and each fixed point of $F$ is a Nash equilibrium of model \eqref{follower2}.
     For any $i\in\mathcal{N}$, denote $$\alpha_i:=\frac{g_{\mathcal{N}}+g_i+\sum_s\widetilde{r}(s)}{2\sum_s\widetilde{w}(s)}$$ and $$\bm{e}_i:=(\underbrace{0,\ldots,0}_i,1,\ldots,1)^{\top}.$$ According to \eqref{expli} and \eqref{seqbs}, we have
     \begin{equation}\label{expli2}
          \begin{aligned}
               \bm{f}_1(\bm{\nu})&=-\frac{1}{2}\bm{\nu}^{\top}\bm{e}_1+\alpha_1(\widetilde{w}(1),\ldots,\widetilde{w}(T))^{\top}\\
                                 &\qquad\qquad\qquad\qquad\qquad\qquad-\frac{1}{2}(\widetilde{r}(1),\ldots,\widetilde{r}(T))^{\top}\\
               \bm{f}_2(\bm{\nu})&=-\frac{1}{2}\Big(\bm{f}_1(\bm{\nu})+\bm{\nu}^{\top}\bm{e}_2\Big)+\alpha_2(\widetilde{w}(1),\ldots,\widetilde{w}(T))^{\top}\\
                                 &\qquad\qquad\qquad\qquad\qquad\qquad-\frac{1}{2}(\widetilde{r}(1),\ldots,\widetilde{r}(T))^{\top}\\
                                 &=-\frac{1}{2}\bm{\nu}^{\top}\Big(\bm{e}_2-\frac{1}{2}\bm{e}_1\Big)+\Big(\alpha_2-\frac{\alpha_1}{2}\Big)\\
                                 &\qquad~~\times(\widetilde{w}(1),\ldots,\widetilde{w}(T))^{\top}-\frac{1}{2^2}(\widetilde{r}(1),\ldots,\widetilde{r}(T))^{\top},\\
                                 &~~\vdots\\
               \bm{f}_n(\bm{\nu})&=-\frac{1}{2}\bm{\nu}^{\top}\Big(\bm{e}_n-\sum_{j=1}^{n-1}\frac{1}{2^{n-j}}\bm{e}_j\Big)+\Big(\alpha_n-\sum_{j=1}^{n-1}\frac{\alpha_j}{2^{n-j}}\Big)\\
               &\qquad~~\times(\widetilde{w}(1),\ldots,\widetilde{w}(T))^{\top}-\frac{1}{2^n}(\widetilde{r}(1),\ldots,\widetilde{r}(T))^{\top}.
          \end{aligned}
     \end{equation}
     % \begin{eqnarray*}
     %      \begin{cases}
     %           \bm{\bar{\nu}}_1=-\frac{1}{2}\bm{\nu}^{\top}\bm{e}_1+\frac{g+g_1}{2g}(\widetilde{w}(1),\ldots,\widetilde{w}(T))^{\top}\\
     %           \bm{\bar{\nu}}_i=-\frac{1}{2}(\sum_{j=1}^{i-1}\bm{\bar{\nu}}_j+\bm{\nu}^{\top}\bm{e}_i)+\frac{g+g_i}{2g}(\widetilde{w}(1),\ldots,\widetilde{w}(T))^{\top}, ~~i=2,\ldots,N,
     %      \end{cases}\nonumber
     % \end{eqnarray*}
     % then we get
     % \begin{equation}\label{expli2}
     %      \begin{aligned}
     %           \bm{f}_i(\bm{\nu})&=-\frac{1}{2}\bm{\nu}^{\top}\Big(\bm{e}_i-\sum_{j=1}^{i-1}\frac{1}{2^{i-j}}\bm{e}_j\Big)+\Big(\alpha_i-\sum_{j=1}^{i-1}\frac{\alpha_j}{2^{i-j}}\Big)\\
     %           &\times(\widetilde{w}(1),\ldots,\widetilde{w}(T))^{\top}-\frac{1}{2^i}(\widetilde{r}(1),\ldots,\widetilde{r}(T))^{\top},\\
     %           &\qquad\qquad\qquad\qquad\qquad\qquad\qquad\qquad\quad i=1,\ldots,n.
     %      \end{aligned}
     % \end{equation}
     Finally, by \eqref{expli2} we can obtain
     \begin{equation}\label{explimap}
          \bm{F}(\bm{\nu})=(\bm{f}_1(\bm{\nu}),\bm{f}_2(\bm{\nu}),\ldots,\bm{f}_n(\bm{\nu}))^{\top}=\mathbf{L}_n\bm{\nu}+\mathbf{M}_n,
     \end{equation}
     where
     \begin{equation*}
          \begin{aligned}
               \mathbf{L}_n:&=-\frac{1}{2}
               \begin{pmatrix}
                    \bm{e}_1^{\top}\\
                    \bm{e}_2^{\top}-\frac{1}{2}\bm{e}_1^{\top}\\
                    \vdots\\
                    \bm{e}_n^{\top}-\sum_{j=1}^{n-1}\frac{1}{2^{n-j}}\bm{e}_j^{\top}\\
               \end{pmatrix}\\
               &=
               \begin{pmatrix}
                    0 & -\frac{1}{2} & -\frac{1}{2} & \cdots & -\frac{1}{2} & -\frac{1}{2}\\
                    0 & \frac{1}{2^2} & -\frac{1}{2^2} & \cdots & -\frac{1}{2^2} & -\frac{1}{2^2}\\
                    \vdots & \vdots & \vdots & \vdots & \vdots & \vdots\\
                    0 & \frac{1}{2^{n-1}} & \frac{2^2-1}{2^{n-1}} & \cdots & \frac{2^{n-2}-1}{2^{n-1}} & -\frac{1}{2^{n-1}}\\
                    0 & \frac{1}{2^n} & \frac{2^2-1}{2^n} & \cdots & \frac{2^{n-2}-1}{2^n} & \frac{2^{n-1}-1}{2^n}\\
               \end{pmatrix}
          \end{aligned}
     \end{equation*}
     and
     \begin{multline*}
          \mathbf{M}_n:=\bigg(\alpha_1,\ldots,\alpha_n-\sum_{j=1}^{n-1}\frac{\alpha_j}{2^{n-j}}\bigg)^{\top}
          (\widetilde{w}(1),\ldots,\widetilde{w}(T))\\
          -\bigg(\frac{1}{2},\ldots,\frac{1}{2^n}\bigg)^{\top}(\widetilde{r}(1),\ldots,\widetilde{r}(T)).
     \end{multline*}

     % Next, we define a sort of matrix norm on $\mathbb{R}^{n\times T}$.
     For above matrix $\mathbf{L}_n\in\mathbb{R}^{n\times n}$, it can be calculated that $\rho(\mathbf{L}_n)<1$, so we can find a constant $\varepsilon_0>0$, such that $\rho(\mathbf{L}_n)+\varepsilon_0<1$.
     According to Lemma \ref{Spectral}, we can find the invertible matrix $\mathbf{S}:=\mathbf{S}_{\mathbf{L}_n,\varepsilon_0}$ such that
     \begin{equation}\label{rho1}
          \|\mathbf{S}^{-1}\mathbf{L}_n\mathbf{S}\|_1\leq\rho(\mathbf{L}_n)+\varepsilon_0.
     \end{equation}
     For any matrix $\mathbf{B}\in\mathbb{R}^{n\times T}$, we set
     \begin{equation}\label{norm}
          \|\mathbf{B}\|_\mathbf{S}:=\|\mathbf{S}^{-1}\mathbf{B}\|_1.
     \end{equation}
     We can verify that i) $\|\mathbf{B}_1-\mathbf{B}_2\|_\mathbf{S}=0$ if and only if $\mathbf{B}_1=\mathbf{B}_2$; ii) $\|\mathbf{B}_1-\mathbf{B}_2\|_\mathbf{S}=\|\mathbf{B}_2-\mathbf{B}_1\|_\mathbf{S}$; iii) $\|\mathbf{B}_1-\mathbf{B}_3\|_\mathbf{S}\leq\|\mathbf{B}_1-\mathbf{B}_2\|_\mathbf{S}+\|\mathbf{B}_2-\mathbf{B}_3\|_\mathbf{S}$.
     % We can verify that $\|\cdot\|_\mathbf{S}$ satisfies the three properties of norm \cite{horn2012matrix},
     Therefore, $\|\cdot\|_\mathbf{S}$ is a metric on $\mathbb{R}^{n\times T}$,
     and $(\mathbb{R}^{n\times T},\|\cdot\|_\mathbf{S})$ is a complete metric space. Further, $\mathbb{H}\subset\mathbb{R}^{n\times T}$ is a closed set, therefore $(\mathbb{H},\|\cdot\|_\mathbf{S})$ is also a complete metric space.

     For any $\bm{\nu}_1, \bm{\nu}_2\in\mathbb{H}$, based on \eqref{explimap}-\eqref{norm}, we have
     \begin{equation*}
          \begin{aligned}
               \|\bm{F}(\bm{\nu}_1)-\bm{F}(\bm{\nu}_2)\|_\mathbf{S} &=\|\mathbf{L}_n\bm{\nu}_1-\mathbf{L}_n\bm{\nu}_2\|_\mathbf{S}\\
                                                 &=\|\mathbf{S}^{-1}\mathbf{L}_n(\bm{\nu}_1-\bm{\nu}_2)\|_1\\
                                                 &=\|\mathbf{S}^{-1}\mathbf{L}_n\mathbf{S}\mathbf{S}^{-1}(\bm{\nu}_1-\bm{\nu}_2)\|_1\\
                                                 &\leq\|\mathbf{S}^{-1}\mathbf{L}_n\mathbf{S}\|_1\|\mathbf{S}^{-1}(\bm{\nu}_1-\bm{\nu}_2)\|_1\\
                                                 &\leq(\rho(\mathbf{L}_n)+\varepsilon_0)\|\bm{\nu}_1-\bm{\nu}_2\|_\mathbf{S},
          \end{aligned}
     \end{equation*}
     where the first inequality uses the norm compatibility of $\|\cdot\|_1$. Since $\rho(\mathbf{L}_n)+\varepsilon_0<1$, the mapping $\bm{F}$ is a contraction mapping on $\mathbb{H}$.
\qed

\bibliographystyle{IEEEtran}
\bibliography{Main}

\begin{thebibliography}{10}
\providecommand{\url}[1]{#1}
\csname url@samestyle\endcsname
\providecommand{\newblock}{\relax}
\providecommand{\bibinfo}[2]{#2}
\providecommand{\BIBentrySTDinterwordspacing}{\spaceskip=0pt\relax}
\providecommand{\BIBentryALTinterwordstretchfactor}{4}
\providecommand{\BIBentryALTinterwordspacing}{\spaceskip=\fontdimen2\font plus
\BIBentryALTinterwordstretchfactor\fontdimen3\font minus
  \fontdimen4\font\relax}
\providecommand{\BIBforeignlanguage}[2]{{%
\expandafter\ifx\csname l@#1\endcsname\relax
\typeout{** WARNING: IEEEtran.bst: No hyphenation pattern has been}%
\typeout{** loaded for the language `#1'. Using the pattern for}%
\typeout{** the default language instead.}%
\else
\language=\csname l@#1\endcsname
\fi
#2}}
\providecommand{\BIBdecl}{\relax}
\BIBdecl

\bibitem{gupta2015gadgets}
S.~Gupta, ``From gadgets to the smart grid,'' \emph{Nature}, vol. 526, no.
  7575, pp. S90--S91, 2015.

\bibitem{beyea2010smart}
J.~Beyea, ``The smart electricity grid and scientific research,''
  \emph{Science}, vol. 328, no. 5981, pp. 979--980, 2010.

\bibitem{smith2022effect}
O.~Smith, O.~Cattell, E.~Farcot, R.~D. O’Dea, and K.~I. Hopcraft, ``The
  effect of renewable energy incorporation on power grid stability and
  resilience,'' \emph{Science Advances}, vol.~8, no.~9, p. eabj6734, 2022.

\bibitem{moslehi2010reliability}
K.~Moslehi and R.~Kumar, ``A reliability perspective of the smart grid,''
  \emph{IEEE Transactions on Smart Grid}, vol.~1, no.~1, pp. 57--64, 2010.

\bibitem{siano2014demand}
P.~Siano, ``Demand response and smart grids—a survey,'' \emph{Renewable and
  Sustainable Energy Reviews}, vol.~30, pp. 461--478, 2014.

\bibitem{liu2023climate}
L.~Liu, G.~He, M.~Wu, G.~Liu, H.~Zhang, Y.~Chen, J.~Shen, and S.~Li, ``Climate
  change impacts on planned supply--demand match in global wind and solar
  energy systems,'' \emph{Nature Energy}, vol.~8, no.~8, pp. 870--880, 2023.

\bibitem{li2021coordinating}
Y.~Li, M.~Han, Z.~Yang, and G.~Li, ``Coordinating flexible demand response and
  renewable uncertainties for scheduling of community integrated energy systems
  with an electric vehicle charging station: A bi-level approach,'' \emph{IEEE
  Transactions on Sustainable Energy}, vol.~12, no.~4, pp. 2321--2331, 2021.

\bibitem{wang2023incentive}
Z.~Wang, B.~Lu, B.~Wang, Y.~Qiu, H.~Shi, B.~Zhang, J.~Li, H.~Li, and W.~Zhao,
  ``Incentive based emergency demand response effectively reduces peak load
  during heatwave without harm to vulnerable groups,'' \emph{Nature
  Communications}, vol.~14, no.~1, p. 6202, 2023.

\bibitem{xu2021hybrid}
B.~Xu, J.~Wang, M.~Guo, J.~Lu, G.~Li, and L.~Han, ``A hybrid demand response
  mechanism based on real-time incentive and real-time pricing,''
  \emph{Energy}, vol. 231, p. 120940, 2021.

\bibitem{jiang2022multi}
T.~Jiang, C.~Chung, P.~Ju, and Y.~Gong, ``A multi-timescale allocation
  algorithm of energy and power for demand response in smart grids: A
  stackelberg game approach,'' \emph{IEEE Transactions on Sustainable Energy},
  vol.~13, no.~3, pp. 1580--1593, 2022.

\bibitem{parizy2018low}
E.~S. Parizy, H.~R. Bahrami, and S.~Choi, ``A low complexity and secure demand
  response technique for peak load reduction,'' \emph{IEEE Transactions on
  Smart Grid}, vol.~10, no.~3, pp. 3259--3268, 2018.

\bibitem{zhong2016stability}
W.~Zhong, R.~Yu, S.~Xie, Y.~Zhang, and D.~K. Yau, ``On stability and robustness
  of demand response in v2g mobile energy networks,'' \emph{IEEE Transactions
  on Smart Grid}, vol.~9, no.~4, pp. 3203--3212, 2016.

\bibitem{sterl2020smart}
S.~Sterl, I.~Vanderkelen, C.~J. Chawanda, D.~Russo, R.~J. Brecha,
  A.~Van~Griensven, N.~P. van Lipzig, and W.~Thiery, ``Smart renewable
  electricity portfolios in west africa,'' \emph{Nature Sustainability},
  vol.~3, no.~9, pp. 710--719, 2020.

\bibitem{cecati2011combined}
C.~Cecati, C.~Citro, and P.~Siano, ``Combined operations of renewable energy
  systems and responsive demand in a smart grid,'' \emph{IEEE Transactions on
  Sustainable Energy}, vol.~2, no.~4, pp. 468--476, 2011.

\bibitem{dai2021real}
Y.~Dai, X.~Sun, Y.~Qi, and M.~Leng, ``A real-time, personalized
  consumption-based pricing scheme for the consumptions of traditional and
  renewable energies,'' \emph{Renewable Energy}, vol. 180, pp. 452--466, 2021.

\bibitem{qian2013demand}
L.~P. Qian, Y.~J.~A. Zhang, J.~Huang, and Y.~Wu, ``Demand response management
  via real-time electricity price control in smart grids,'' \emph{IEEE Journal
  on Selected Areas in Communications}, vol.~31, no.~7, pp. 1268--1280, 2013.

\bibitem{li2020real}
H.~Li, Z.~Wan, and H.~He, ``Real-time residential demand response,'' \emph{IEEE
  Transactions on Smart Grid}, vol.~11, no.~5, pp. 4144--4154, 2020.

\bibitem{mahmoudi2014wind}
N.~Mahmoudi, T.~K. Saha, and M.~Eghbal, ``Wind power offering strategy in
  day-ahead markets: Employing demand response in a two-stage plan,''
  \emph{IEEE Transactions on Power Systems}, vol.~30, no.~4, pp. 1888--1896,
  2014.

\bibitem{ding2020tracking}
T.~Ding, M.~Qu, N.~Amjady, F.~Wang, R.~Bo, and M.~Shahidehpour, ``Tracking
  equilibrium point under real-time price-based residential demand response,''
  \emph{IEEE Transactions on Smart Grid}, vol.~12, no.~3, pp. 2736--2740, 2020.

\bibitem{samadi2014real}
P.~Samadi, H.~Mohsenian-Rad, V.~W. Wong, and R.~Schober, ``Real-time pricing
  for demand response based on stochastic approximation,'' \emph{IEEE
  Transactions on Smart Grid}, vol.~5, no.~2, pp. 789--798, 2014.

\bibitem{namerikawa2015real}
T.~Namerikawa, N.~Okubo, R.~Sato, Y.~Okawa, and M.~Ono, ``Real-time pricing
  mechanism for electricity market with built-in incentive for participation,''
  \emph{IEEE Transactions on Smart Grid}, vol.~6, no.~6, pp. 2714--2724, 2015.

\bibitem{atzeni2014noncooperative}
I.~Atzeni, L.~G. Ord{\'o}{\~n}ez, G.~Scutari, D.~P. Palomar, and J.~R.
  Fonollosa, ``Noncooperative day-ahead bidding strategies for demand-side
  expected cost minimization with real-time adjustments: A gnep approach,''
  \emph{IEEE Transactions on Signal Processing}, vol.~62, no.~9, pp.
  2397--2412, 2014.

\bibitem{baniasadi2018optimal}
A.~Baniasadi, D.~Habibi, O.~Bass, and M.~A. Masoum, ``Optimal real-time
  residential thermal energy management for peak-load shifting with
  experimental verification,'' \emph{IEEE Transactions on Smart Grid}, vol.~10,
  no.~5, pp. 5587--5599, 2018.

\bibitem{choi2017practical}
S.~Choi, ``Practical coordination between day-ahead and real-time optimization
  for economic and stable operation of distribution systems,'' \emph{IEEE
  Transactions on Power Systems}, vol.~33, no.~4, pp. 4475--4487, 2017.

\bibitem{galeotti2010network}
A.~Galeotti, S.~Goyal, M.~O. Jackson, F.~Vega-Redondo, and L.~Yariv, ``Network
  games,'' \emph{The Review of Economic Studies}, vol.~77, no.~1, pp. 218--244,
  2010.

\bibitem{chen2023convergence}
G.~Chen and Y.~Yu, ``Convergence analysis and strategy control of evolutionary
  games with imitation rule on toroidal grid,'' \emph{IEEE Transactions on
  Automatic Control}, vol.~68, no.~12, pp. 8185--8192, 2023.

\bibitem{cheng2023evolutionary}
J.~Cheng, W.~Mei, W.~Su, and G.~Chen, ``Evolutionary games on networks: Phase
  transition, quasi-equilibrium, and mathematical principles,'' \emph{Physica
  A: Statistical Mechanics and its Applications}, vol. 611, p. 128447, 2023.

\bibitem{riehl2016towards}
J.~R. Riehl and M.~Cao, ``Towards optimal control of evolutionary games on
  networks,'' \emph{IEEE Transactions on Automatic Control}, vol.~62, no.~1,
  pp. 458--462, 2016.

\bibitem{madeo2014game}
D.~Madeo and C.~Mocenni, ``Game interactions and dynamics on networked
  populations,'' \emph{IEEE Transactions on Automatic Control}, vol.~60, no.~7,
  pp. 1801--1810, 2014.

\bibitem{tan2016analysis}
S.~Tan, Y.~Wang, and J.~L{\"u}, ``Analysis and control of networked game
  dynamics via a microscopic deterministic approach,'' \emph{IEEE Transactions
  on Automatic Control}, vol.~61, no.~12, pp. 4118--4124, 2016.

\bibitem{moon2016linear}
J.~Moon and T.~Ba{\c{s}}ar, ``Linear quadratic risk-sensitive and robust mean
  field games,'' \emph{IEEE Transactions on Automatic Control}, vol.~62, no.~3,
  pp. 1062--1077, 2016.

\bibitem{morimoto2015subsidy}
T.~Morimoto, T.~Kanazawa, and T.~Ushio, ``Subsidy-based control of
  heterogeneous multiagent systems modeled by replicator dynamics,'' \emph{IEEE
  Transactions on Automatic Control}, vol.~61, no.~10, pp. 3158--3163, 2015.

\bibitem{guo2013algebraic}
P.~Guo, Y.~Wang, and H.~Li, ``Algebraic formulation and strategy optimization
  for a class of evolutionary networked games via semi-tensor product method,''
  \emph{Automatica}, vol.~49, no.~11, pp. 3384--3389, 2013.

\bibitem{cheng2016decomposed}
D.~Cheng, T.~Liu, K.~Zhang, and H.~Qi, ``On decomposed subspaces of finite
  games,'' \emph{IEEE Transactions on Automatic Control}, vol.~61, no.~11, pp.
  3651--3656, 2016.

\bibitem{zhu2016evolutionary}
B.~Zhu, X.~Xia, and Z.~Wu, ``Evolutionary game theoretic demand-side management
  and control for a class of networked smart grid,'' \emph{Automatica},
  vol.~70, pp. 94--100, 2016.

\bibitem{etesami2018stochastic}
S.~R. Etesami, W.~Saad, N.~B. Mandayam, and H.~V. Poor, ``Stochastic games for
  the smart grid energy management with prospect prosumers,'' \emph{IEEE
  Transactions on Automatic Control}, vol.~63, no.~8, pp. 2327--2342, 2018.

\bibitem{mohsenian2010autonomous}
A.-H. Mohsenian-Rad, V.~W. Wong, J.~Jatskevich, R.~Schober, and A.~Leon-Garcia,
  ``Autonomous demand-side management based on game-theoretic energy
  consumption scheduling for the future smart grid,'' \emph{IEEE Transactions
  on Smart Grid}, vol.~1, no.~3, pp. 320--331, 2010.

\bibitem{bayram2014unsplittable}
I.~S. Bayram, G.~Michailidis, and M.~Devetsikiotis, ``Unsplittable load
  balancing in a network of charging stations under qos guarantees,''
  \emph{IEEE Transactions on Smart Grid}, vol.~6, no.~3, pp. 1292--1302, 2014.

\bibitem{li2022noncooperative}
H.~Li, Z.~Ren, A.~Trivedi, P.~P. Verma, D.~Srinivasan, and W.~Li, ``A
  noncooperative game-based approach for microgrid planning considering
  existing interconnected and clustered microgrids on an island,'' \emph{IEEE
  Transactions on Sustainable Energy}, vol.~13, no.~4, pp. 2064--2078, 2022.

\bibitem{maharjan2013dependable}
S.~Maharjan, Q.~Zhu, Y.~Zhang, S.~Gjessing, and T.~Basar, ``Dependable demand
  response management in the smart grid: A stackelberg game approach,''
  \emph{IEEE Transactions on Smart Grid}, vol.~4, no.~1, pp. 120--132, 2013.

\bibitem{yu2015real}
M.~Yu and S.~H. Hong, ``A real-time demand-response algorithm for smart grids:
  A stackelberg game approach,'' \emph{IEEE Transactions on Smart Grid},
  vol.~7, no.~2, pp. 879--888, 2015.

\bibitem{yan2020distribution}
M.~Yan, M.~Shahidehpour, A.~Paaso, L.~Zhang, A.~Alabdulwahab, and A.~Abusorrah,
  ``Distribution network-constrained optimization of peer-to-peer transactive
  energy trading among multi-microgrids,'' \emph{IEEE Transactions on Smart
  Grid}, vol.~12, no.~2, pp. 1033--1047, 2020.

\bibitem{li2019stackelberg}
J.~Li, G.~Ma, T.~Li, W.~Chen, and Y.~Gu, ``A stackelberg game approach for
  demand response management of multi-microgrids with overlapping sales
  areas,'' \emph{Science China Information Sciences}, vol.~62, pp. 1--13, 2019.

\bibitem{bacsar1998dynamic}
T.~Ba{\c{s}}ar and G.~J. Olsder, \emph{Dynamic noncooperative game
  theory}.\hskip 1em plus 0.5em minus 0.4em\relax SIAM, 1998.

\bibitem{borwein2006convex}
J.~Borwein and A.~Lewis, \emph{Convex Analysis}.\hskip 1em plus 0.5em minus
  0.4em\relax Springer, 2006.

\bibitem{kantorovich2016functional}
L.~V. Kantorovich and G.~P. Akilov, \emph{Functional analysis}.\hskip 1em plus
  0.5em minus 0.4em\relax Elsevier, 2016.

\bibitem{rosen1965existence}
J.~B. Rosen, ``Existence and uniqueness of equilibrium points for concave
  n-person games,'' \emph{Econometrica: Journal of the Econometric Society},
  pp. 520--534, 1965.

\bibitem{laub2004matrix}
A.~J. Laub, \emph{Matrix analysis for scientists and engineers}.\hskip 1em plus
  0.5em minus 0.4em\relax SIAM, 2004.

\bibitem{gould2005numerical}
N.~Gould, D.~Orban, and P.~Toint, ``Numerical methods for large-scale nonlinear
  optimization,'' \emph{Acta Numerica}, vol.~14, pp. 299--361, 2005.

\bibitem{hager1989updating}
W.~W. Hager, ``Updating the inverse of a matrix,'' \emph{SIAM Review}, vol.~31,
  no.~2, pp. 221--239, 1989.

\end{thebibliography}

\end{document}